\documentclass[reqno]{amsart}

\usepackage[centering]{geometry}

\usepackage{mathtools}
\usepackage{mathrsfs}
\usepackage{amsmath}
\usepackage{bm}
\usepackage[dvipsnames]{xcolor}
\usepackage
[colorlinks=true,linkcolor=Maroon,citecolor=OliveGreen]
{hyperref}

\usepackage[capitalize]{cleveref} 
\crefname{equation}{}{}

\usepackage[norefs, nocites]{refcheck}

\makeatletter
\newcommand{\refcheckize}[1]{%
  \expandafter\let\csname @@\string#1\endcsname#1%
  \expandafter\DeclareRobustCommand\csname relax\string#1\endcsname[1]{%
    \csname @@\string#1\endcsname{##1}\wrtusdrf{##1}}%
  \expandafter\let\expandafter#1\csname relax\string#1\endcsname
}
\makeatother

\newtheorem{theorem}{Theorem}[section]
\newtheorem{lemma}[theorem]{Lemma}
\newtheorem{proposition}[theorem]{Proposition}
\newtheorem{corollary}[theorem]{Corollary}

\theoremstyle{definition}
\newtheorem{remark}[theorem]{Remark}

\numberwithin{equation}{section}

\usepackage{amssymb}

\newcommand{\eps}{\epsilon}

\newcommand{\wt}{\widetilde}
\newcommand{\sm}{\smallsetminus}

\def\C{\mathrm C}
\def\N{\mathrm N}

\def\End{\mathrm{End}}
\def\F{\mathbf F}
\def\dim{\mathrm{dim }}
\def\Z{\mathrm Z}
\def\GL{\mathrm{GL}}

\def\Gal{\mathrm{Gal}}
\def\SU{\mathrm{SU}}

\def\GU{\mathrm{GU}}
\def\Sp{\mathrm{Sp}}

\def\SL{\mathrm{SL}}

\def\Or{\mathrm{O}}
\def\SO{\mathrm{SO}}

\def\G{\mathbf G}

\def\GammaL{\mathrm{\Gamma L}}
\def\det{\operatorname{det}}

\def\min{\mathrm{min}}
\def\AGamL{\mathrm{A\Gamma L}}
\def\GamL{\mathrm{\Gamma L}}
\def\Out{\mathrm{Out}}
\def\Aut{\mathrm{Aut}}
\def\Ker{\mathrm{Ker}}

\def\Sym{\mathrm{Sym}}
\def\End{\mathrm{End}}

\newcommand{\ind}[1]{\left[#1\right]}
\newcommand\floor[1]{\left\lfloor{#1}\right\rfloor}

\newcommand{\gen}[1]{\ensuremath{\langle #1\rangle}}

\usepackage{todonotes}

\begin{document}

\author{Daniele Garzoni}
\address{Daniele Garzoni, Department of Mathematics, University of Southern California, Los Angeles, CA 90089-2532,
	USA}
\email{garzoni@usc.edu}

\title{Derangements in non-Frobenius groups}

\begin{abstract}
We prove that if $G$ is a transitive permutation group of sufficiently large degree $n$, then either $G$ is primitive and Frobenius, or the proportion of derangements in $G$ is larger than $1/(2n^{1/2})$.  This is sharp, generalizes substantially bounds of Cameron--Cohen and Guralnick--Wan, and settles conjectures of Guralnick--Tiep and Bailey--Cameron--Giudici--Royle in large degree. We also give an application to coverings of varieties over finite fields. 
\end{abstract}

\maketitle

\section{Introduction}

Given a finite transitive permutation group $G$ on a set $\Omega$ of size $n\ge 2$, an element $g\in G$ is called a \textit{derangement} if it acts without fixed points on $\Omega$. 

Chebotarev density theorem establishes a direct link between derangements and rational points of varieties over finite fields. In particular, derangements arise naturally in arithmetic geometry and were apparent already in the work of Frobenius in the nineteenth century. Derangements have other applications, for example to problems of random generation (see \cite{luczak1993random,eberhard2016permutations,lucchini2018chebotarev,garzoni2023probability}) and to the structure of Brauer groups (see \cite{fein1981relative}).

Denote by $\delta(G,\Omega)$, or simply $\delta(G)$, the proportion of derangements of $G$ on $\Omega$. A lemma going back to Jordan asserts that $\delta(G)>0$, and giving lower bounds for $\delta(G)$ has been a central problem in group theory in the past three decades;
see \Cref{t:arithmetic}, \Cref{sec:proof} and \cite{burness2016classical} for context related to this problem.

Answering a question of Lenstra motivated by the number field sieve (see \cite[Section 9]{buhler_lenstra_pomerance}), Cameron--Cohen \cite{cameron_cohen} showed that $\delta(G)\ge 1/n$; more precisely, if $n\ge 3$ then one of the following holds:
\begin{itemize}
	 \item[(i)] $G$ is a Frobenius group of order $n(n-1)$ and $\delta(G)=1/n$.
    \item[(ii)] $\delta(G)>1/n$.
\end{itemize}
The proof of this result is very short and elementary (see also \cite{boston1993proportion}, which further simplifies the proof). See below for the definition of a Frobenius group. Later, Guralnick--Wan \cite{guralnick1997wan} showed that if $n\ge 7$ then one of the following holds:
\begin{itemize}
	\item[(i)] $G$ is a Frobenius group of order $n(n-1)/a$ and $\delta(G)=a/n$, where $a\in \{1,2\}$.
    \item[(ii)] $\delta(G)>2/n$.
\end{itemize}
The proof of this result is much harder than that of \cite{cameron_cohen} -- it requires, besides some nontrivial elementary arguments,  the Classification of Finite Simple Groups.

The groups attaining the bounds $\delta(G)=1/n$ and $2/n$ are very special: they are primitive Frobenius groups.  We recall that a transitive permutation group $G$ is called \textit{Frobenius} if some nontrivial element fixes a point, and only the identity fixes at least two points. A classical result of Frobenius, based on character theory, states that $G$ contains a regular normal subgroup, which is furthermore nilpotent by a result of Thompson.

If $G$ is Frobenius, then $|G|=n(n-1)/a$ for a divisor $a$ of $n-1$, and $\delta(G)=a/n$. There are infinitely many examples attaining this bound; for instance, for every prime power $n$ and every divisor $a$ of $n-1$, consider $G=\F_n \rtimes G_0$ where $G_0\le \F_n^\times$ has order $(n-1)/a$. 
In particular, we find infinitely many examples with
\[
\delta(G)=\frac{1}{n},\frac{2}{n}, \frac{3}{n}, \frac{4}{n} \ldots
\]
Given the exceptionality of Frobenius groups, it is natural to ask whether a considerably stronger lower bound holds in the other cases. 
The purpose of this paper is to answer this question in large degree, by showing that a bound of approximately $1/(2n^{1/2})$ holds. What is more, it is sufficient to exclude groups that are both primitive and Frobenius in order to get this improvement.
Let
\begin{equation}
	\label{eq:def_g}
	g(n):=\frac{n^{1/2}+1}{2n}\sim \frac{1}{2n^{1/2}}. 
\end{equation}

\begin{theorem}
\label{t:main}
Let $G$ be a finite transitive permutation group of sufficiently large degree $n$. Then one of the following holds:
\begin{enumerate}
    \item $G$ is primitive and Frobenius.
    \item $\delta(G)\ge g(n)$.
\end{enumerate}
\end{theorem}
The crux of this paper is the case of primitive affine groups; see \Cref{sec:proof} and \Cref{t:affine} for details.

There are primitive non-Frobenius groups attaining equality in \Cref{t:main}(2); take
 $G=\F_n\rtimes G_0$, where $n$ is a square, $G_0=\GL_1(n)\rtimes \gen\sigma$ and $\sigma$ is the $n^{1/2}$-th power map (see \Cref{ex:bound_sharp}).

\Cref{t:main} generalizes substantially the aforementioned bounds of \cite{cameron_cohen,guralnick1997wan}. What is more, in \cite[p. 272]{guralnick_tiep_eigenvalue1}, Guralnick--Tiep conjectured that $\delta(G)\ge 1/n^{1/2}$ if $G$ is primitive affine and non-Frobenius. \Cref{t:main}, in particular, confirms (up to a necessary minor amendment) this conjecture in large degree, showing also that the primitivity assumption is not needed.

Let $D(G)$ be the subgroup of $G$ generated by all derangements. Bailey--Cameron--Giudici--Royle \cite[Conjecture 1]{bailey2021groups} conjectured that if $G$ is not primitive Frobenius then $|G:D(G)|\le n^{1/2}-1$, and reduced the statement to the case of primitive affine groups. We clearly have
\begin{equation}
\label{eq:subgroup_conjecture}
|G:D(G)|\le \frac{1}{\delta(G)},
\end{equation}
hence by \Cref{t:main} we get  $|G:D(G)|\le 1/g(n)\sim 2n^{1/2}$ for $n$ large. In fact, in most cases we will prove $\delta(G)>1/(n^{1/2}-1)$, so the conjecture will follow at once by \Cref{eq:subgroup_conjecture}. In the remaining cases, the bound will be easily obtained by inspection of the proof of \Cref{t:main}, so we deduce \cite[Conjecture 1]{bailey2021groups} in large degree.

\begin{corollary}
    \label{t:subgroup}
Let $G$ be a finite transitive permutation group of sufficiently large degree $n$. Then one of the following holds:
\begin{enumerate}
    \item $G$ is primitive and Frobenius.
    \item $|G:D(G)|\le n^{1/2}-1$.
\end{enumerate}
\end{corollary}

Next, we state a consequence of \Cref{t:main} that highlights the aforementioned connection between derangements and arithmetic.

\begin{corollary} 
\label{t:arithmetic}
 Let $\pi\colon X \to Y$ be a separable morphism of sufficiently large degree $n$ between integral normal quasi-projective varieties of dimension $d$ over a finite field $\F_\ell$. Assume that the Galois closure of $\pi$ is geometrically integral, and that the monodromy group of $\pi$ is not primitive and Frobenius. Then,
 \[
|\pi(X(\F_\ell))| \le \big(1- g(n)\big) \ell^d + O_\pi(\ell^{d-1/2}).
 \]
    \end{corollary}

  If $\pi$ is defined by a polynomial $f(x)\in \F_\ell(Y)[x]$,  then the condition that the monodromy group is primitive and Frobenius can be phrased field theoretically as follows: $\F_\ell(X)/\F_\ell(Y)$ is minimal, and the Galois closure of $\F_\ell(X)/\F_\ell(Y)$ is not $\F_\ell(X)$, and is generated by any two roots of $f$. Therefore the bound in \Cref{t:arithmetic} applies whenever this (very restrictive) condition does not hold.

    If the Galois closure of $\pi$ (namely, the normalization of $X$ in the Galois closure of $\F_\ell(X)/\F_\ell(Y)$) is not geometrically integral, then one must more generally count derangements in a coset (indeed this is one of the motivations of \cite{guralnick1997wan}, which addresses the case of curves).
    
  With regard to \Cref{t:main}, it seems worth noting that one gets almost the same bound (approximately $1/(60n^{1/2})$) just by excluding primitive Frobenius subgroups of $\AGamL_1(n)$, as opposed to all primitive Frobenius groups. Let
  \begin{equation}
  	\label{eq:def_f}
  	f(n):= \frac{n^{1/2}+1}{60n}\sim \frac{1}{60n^{1/2}}.
  \end{equation}

\begin{theorem}
	\label{t:main_2}
	Let $G$ be a finite transitive permutation group of sufficiently large degree $n$. Then one of the following holds:
	\begin{enumerate}
		\item $G$ is a primitive Frobenius subgroup of $\AGamL_1(n)$.
		\item $\delta(G)\ge f(n)$.
	\end{enumerate}
\end{theorem}
There are primitive (Frobenius) groups that do not lie in (1) and that attain equality in (2), of the form $G=\F_n^2\rtimes G_0$, where  $\SL_2(5)\trianglelefteq G_0$; see \Cref{rem:SL_2(5)_frobenius}. 

   \subsection{About the proof} 
    \label{sec:proof}
    We prove  an elementary reduction of \Cref{t:main} to primitive groups, see \Cref{l:reduction_primitive}.
    
Luczak--Pyber \cite{luczak1993random} and Fulman--Guralnick 
\cite{FG4} proved the so-called \textit{Boston--Shalev conjecture}, asserting that if $G$ is a finite simple transitive permutation group of degree $n$, then $\delta(G)\ge \epsilon$ for an absolute constant $\epsilon>0$. See also \cite{eberhard2023garzoni} for a conjugacy-class version. A consequence of this result is the following (see \cite{FG}):

\begin{theorem}(Luczak--Pyber, Fulman--Guralnick)
\label{t:luczak_guralnick}
Let $G$ be a primitive non-affine permutation group of degree $n$. Then $\delta(G)\ge \epsilon/\log(n)$ for an absolute constant $\epsilon>0$.     
\end{theorem}

Of course, if $n$ is sufficiently large then $\eps/\log(n) > g(n)$, and so in view of \Cref{t:luczak_guralnick} and \Cref{l:reduction_primitive}, it remains to prove \Cref{t:main} in the case where $G$ is primitive affine, so $G=V\rtimes G_0$ where $V$ is elementary abelian and $G_0\le \GL(V)$ is irreducible. 

In this case, there is an amusing  phenomenon -- $G$ has many elements fixing no point of $V=\Omega$ if and only if $G_0$ has many elements fixing a point of $V\sm \{0\}$ (see \Cref{lemma_delta_alpha}).

In particular, denoting by $\alpha(G_0,V)$, or simply $\alpha(G_0)$, the proportion of elements of $G_0$ having eigenvalue $1$ on $V$, our task is to lower bound $\alpha(G_0)$. This is a well-studied problem, especially for quasisimple groups; see for example \cite{guralnick_tiep_eigenvalue1,suprunenko2007fixed,cullinan2021unisingular,spiga2017number,neumannpraeger} and the references therein. Many of these references are concerned with the so-called \textit{unisingular} representations, which is to say, representations where  $\alpha(G_0)=1$.

The following theorem about irreducible linear groups implies \Cref{t:main,t:subgroup,t:main_2} for primitive affine groups. We work over any finite field, not necessarily of prime order. Let $A(G_0,V)$ or $A(G_0)$ be the subgroup of $G_0$ generated by all elements with eigenvalue $1$, and let
\begin{equation}
	\label{eq:def_h}
h(n):=\frac{n^{1/2}+2}{2(n-1)}\sim \frac{1}{2n^{1/2}}.
\end{equation}

\begin{theorem}
\label{t:affine}
Let $d$ be a positive integer and $q$ be a prime power, with $qd$ sufficiently large. Let $V\cong\F_q^d$ and $G\le \GL_d(q)$ be irreducible, and $r$ be maximal so that $G\le \GamL_{d/r}(q^r)$. Then one of the following holds:
\begin{enumerate}
    \item $d=r$ and $G$ acts semiregularly on $V\sm\{0\}$.
    \item $G$ acts semiregularly on $V\sm\{0\}$, and $\alpha(G)\ge 1/(60(q^{d/2}-1))$ and $\delta(V\rtimes G)\ge f(q^d)$.
    \item $\alpha(G)\ge h(q^d)$, $\delta(V\rtimes G) \ge g(q^d)$, and $|A(G)|/|G|\ge 1/(q^{d/2}-1)$.
\end{enumerate}
\end{theorem}

(In (2) and (3), $\delta(V\rtimes G)=\delta(V\rtimes G,V)$ refers to the affine action on $V$.) We already remarked that the equalities in the bounds for $\delta(V\rtimes G)$ can be attained. The same is true for the bounds for $\alpha(G)$, see \Cref{ex:bound_sharp,rem:SL_2(5)_frobenius}.

In most cases (e.g., for $d/r\ge 3$) we will prove the bound for $\alpha(G)$, and the required bound for $\delta(V\rtimes G)$ will follow immediately.

In order to prove \Cref{t:affine}, we will reduce the problem to the case of primitive linear groups and exploit the powerful structure theory of these groups, see \Cref{sec_primitive_linear}. The reduction is elementary (\Cref{l:general_imprimit}), so assume $G$ is primitive. At this point, after some further reductions, we will focus on the generalized Fitting subgroup $L$ (whose index in $G$ is relatively small), and we will count elements with eigenvalue $1$ in $L$. Assuming for simplicity that $L$ acts absolutely irreducibly, we have $L=G_1\circ \cdots \circ G_t$, where $G_i$ is either central or extraspecial or a power of a quasisimple group, and $V=V_1\otimes \cdots \otimes V_t$. This will essentially reduce the problem to lower bounding $\alpha(H)$ when $H\le \GamL_1(q)$ (\Cref{sec:GammaL1}), or $H\le \GamL_2(q)$ (\Cref{sec:GamL2}), or $H$  extraspecial or quasisimple (\Cref{sec:extraspecial,sec:quasisimple}). 

A nice feature of the proof is that we will be able to prove \Cref{t:affine} only by working with large extraspecial or quasisimple groups (and in many cases, bounds much stronger than we need hold). Note, however, that in the central product $L=G_1\circ \cdots \circ G_t$, there may certainly be factors that have small order, and to which our bounds do not apply. In order to handle this issue, we will group together all the factors of small order (thereby paying a constant); what we gain in the factors of large order will be enough to compensate.

We wonder whether, in the notation of \Cref{t:affine},  the bound $\alpha(G)\gg 1/q^r$ holds. There are many examples where a matching upper bound holds, and the lower bound holds if $d/r\le 2$ (by \Cref{t:affine}) and for many quasisimple groups in defining characteristic, see \Cref{l:codimension_1} and also \cite{guralnick_tiep_eigenvalue1}. This problem could be considered as an analogue of the Boston--Shalev conjecture for affine groups over fields of bounded size.

\subsection*{Notation} We write $f=O(g)$ or $f\ll g$ if there exists a positive absolute constant $C$ such that $f\le Cg$. We number equations by sections, so for example \Cref{t:main} is not interchangeable with \eqref{eq:def_g}.

\subsection*{Acknowledgements} I thank Bob Guralnick for suggesting the proof of \Cref{l:codimension_1} and for  interesting discussions, and Michael Larsen for an explanation on the Lang--Weyl estimate for Suzuki and Ree groups.

\section{Reduction to primitive groups}

\begin{lemma}
\label{l:reduction_primitive}
    Assume that \Cref{t:main} holds for primitive groups. Then, $\delta(G)>g(n)$ for every imprimitive permutation group $G$  of sufficiently large degree $n$.
\end{lemma}
In fact, the proof will show that if $G$ is imprimitive then $\delta(G)>1/(1.5n^{1/2})$.

\begin{proof}
 Let $\Omega= \Omega_1\cup \cdots \cup \Omega_t$ be a maximal system of imprimitivity, with $|\Omega_i|=m\ge 2$ and $t\ge 2$, let $\Delta=\{\Omega_1, \ldots, \Omega_t\}$ and $\rho\colon G \to \Sym(\Delta)\cong S_t$.  Then, $G^\rho$ is primitive.

By \cite{cameron_cohen}, we have $\delta(G,\Omega)\ge \delta(G^\rho,\Delta)\ge 1/t$. If $t$ is bounded this is $\gg 1$ and we are done, so suppose that $t$ is large.

\textbf{Case 1:} $G^\rho$ is not Frobenius. Then, by \Cref{t:main} $\delta(G,\Omega)\ge \delta(G^\rho,\Delta) \ge g(t)$. Since $n=mt\ge 2t$, we have $g(t) > g(n)$ and we are done. 

\textbf{Case 2:} $G^\rho$ is Frobenius. We have $|G^\rho|= t(t-1)/a$ for a divisor $a$ of $t-1$, and $\delta(G^\rho, \Delta) = a/t$.

Assume first $t/a < 1.5n^{1/2}$. Then $a/t > 1/(1.5n^{1/2}) > g(n)$ and we are done. Assume then $t/a \ge 1.5n^{1/2}$, i.e., $t\ge 2.25ma^2$. For each $i=1, \ldots, t$, let $H_i=\text{Stab}_G(\Omega_i)$, so $H_i$ acts transitively on $\Omega_i$. Since $G^\rho$ is Frobenius, a nontrivial element of $G^\rho$ fixes at most one point. In particular,
\[
\delta(G,\Omega) \ge \sum_{i=1}^t \left( \frac{\delta(H_i, \Omega_i)}{|G:H_i|} - \frac{1}{|G^\rho|} \right) = \delta(H_1, \Omega_1)- \frac{a}{t-1}.
\]
Now fix $0<\eps < 1-1/1.1 \le 1-1/(1.1a)$ (e.g., $\eps=0.09$) and take $t$ large so that $(t-1)/t> 1-\eps$. By \cite{cameron_cohen}, we have $\delta(H_1, \Omega_1)\ge 1/m$, and the assumption $t\ge 2.2ma^2$ implies 
\begin{align*}
\frac{1}{m} - \frac{a}{t-1} &> \frac{1}{m} - \frac{a}{(1-\eps)t} 
\ge \frac{1}{m}\left( 1 - \frac{1}{(1-\eps)2.2a} \right) \\
&\quad > \frac{1}{2m} \\
&\quad\quad > \frac{1.4}{2n^{1/2}}\\
&\quad\quad\quad > g(n),
\end{align*}
which concludes the proof.    
\end{proof}

\section{Preliminaries for affine groups}
\label{sec:affine}

Let $G$ be a finite group and let $V$ be an $\F_qG$-module of dimension $d<\infty$. Denote by $\pi$ the permutation character of $G$ acting on $V$, and define $\eta(G)=\eta(G,V)$ as the inverse of the harmonic mean of $\pi$, i.e.,
\begin{equation}
    \eta(G)=\frac{1}{|G|} \sum_{g\in G} \frac{1}{\pi(g)}.
\end{equation}

\begin{lemma}
\label{derangements_eta}
$\delta(V\rtimes G)=1-\eta(G)$.
\end{lemma}

\begin{proof}
Let $gv\in V\rtimes G$. We have $u^{gv}= ug+v$, so $u^{gv}=u$ if and only if $v=u-ug$. It follows that $gv$ is a derangement if and only if $v\notin [g,V]$. Let $c(g):=\dim(\C_V(g))$, so $\pi(g)=q^{c(g)}$. Since $[g,V]$ is a subspace of $V$ of dimension $d-c(g)$, and counting derangements  by summing over all $g\in G$ and all $v\not\in [g,V]$, we have
\[
\delta(V\rtimes G) = \frac{1}{|G|} \sum_{g\in G} \frac{q^d - q^{d-c(g)}}{q^d} = 1- \eta(G)
\]
as wanted.
\end{proof}

Recall that $\alpha(G)$ denotes the proportion of elements of $G$ with eigenvalue $1$ on $V$, and
$A(G)$ denotes the subgroup of $G$ generated by all elements with eigenvalue $1$.

\begin{lemma}
\label{lemma_delta_alpha}
 \[
\left(1-\frac{1}{q}\right)\alpha(G) \le \delta(V\rtimes G)\le \alpha(G).
\]
\end{lemma}

\begin{proof}
Writing $\alpha=\alpha(G)$ and $\eta=\eta(G)$, note that
\[
\frac{1}{|G|}\left(1-\alpha\right)|G|\le \eta \le \frac{1}{|G|}\left( \frac{\alpha|G|}{q}+(1-\alpha)|G|\right) = 1-\left(1-\frac{1}{q}\right)\alpha.
\]
Now just apply Lemma \ref{derangements_eta}.
\end{proof}

In particular, as already noticed in the introduction, in most cases in order to prove \Cref{t:affine} it will be sufficient to bound $\alpha(G)$. 

The following is \cite[Proposition 3.1]{bailey2021groups}.

\begin{lemma}
	\label{l:subgroup_eigenvalue1}
$|V\rtimes G : D(V\rtimes G)| = |G:A(G)|$. 
\end{lemma}

\begin{proof}
We have $V\le D(V\rtimes G)$, and as noted in the proof of \Cref{derangements_eta}, $gv$ is a derangement if and only if $v\not\in [g,V]$; therefore for $g\in G$ there exists $v$ such that $gv$ is a derangement if and only if $g\in A(G)$.
	\end{proof}

\section{Proof of \Cref{t:affine} modulo preliminary results}
\label{sec_primitive_linear}

Here we reduce the proof of \Cref{t:affine} to the following three results, that we will prove in \Cref{sec:GammaL1,sec:GamL2,sec:extraspecial,sec:quasisimple}.

\begin{proposition}
\label{prop:GamL_1}
\Cref{t:affine} holds if $d=r$, i.e., $G\le \GamL_1(q^d)$.
\end{proposition}

\begin{proposition}
\label{prop:GamL_2}
\Cref{t:affine} holds if $G$ is primitive and $d/r=2$, i.e., $G\le \GamL_2(q^{d/2})$.
\end{proposition}

\begin{proposition}
\label{prop:combination}
   There exists an absolute constant $C>0$ such that if $d\ge 3$,  $G\le \GL_d(q)$ is an absolutely irreducible quasisimple or extraspecial group of order at least $C$, and $Z\le \F_q^\times$, then  $\alpha(ZG)> 2|O|\log(q)q^{-d/2}$, where  $O=\Out(G)$ if $G$ is quasisimple and $O=\Sp_{2r}(s)$ if $G$ is extraspecial of order $s^{2r+1}$.
   \end{proposition}

The factor $2$ in the inequality in \Cref{prop:combination} will be useful in the proof of \Cref{t:affine} for technical reasons. 

In \Cref{prop:GamL_1}, we do not assume that $G$ is primitive because it will be convenient to quote the result in that form. In every case, let us show that the imprimitive case follows from the primitive one.

\begin{lemma}
	\label{l:general_imprimit}
If \Cref{t:affine} holds when $G$ is a primitive linear group, then it holds when $G$ is an imprimitive linear group.
\end{lemma}

\begin{proof}
Let $G\le \GL_m(q)\wr S_t$  preserve the decomposition $V=V_1\oplus \cdots \oplus V_t$, with $t\ge 2$, permuting transitively the factors, and assume the stabilizer $H$ of $V_1$ acts primitively on $V_1$. Clearly $|A(G)|/|G|\ge \alpha(G)$, and in all but one cases the bound for $\alpha(G)$ will be enough also for $|A(G)|/|G|$; we will not repeat this every time. 

 Note $\alpha(G)\ge \alpha(H,V_1)/t$.
	If $H$ does not induce a semiregular group on $V_1\sm\{0\}$, then $\alpha(H,V_1)\ge h(q^m) > 1/(2q^{m/2})$ and so $\alpha(G)> 1/(2tq^{m/2})\gg 1/q^{d/3}$ and we are done (recall \Cref{lemma_delta_alpha}). Assume then that $H$ induces a semiregular group on $V_1\sm\{0\}$; then 
	$s:=|H^{V_1}|\le q^m-1$ and $\alpha(H,V_1)=1/s$, and so $\alpha(G)\ge 1/(t(q^m-1))$. If $m\le d/3$ this is $\gg q^{d/3}$, so assume $m=d/2$. 
	
	For $i=1,2$, let $N_i$ be the subgroup of $H$ acting trivially on $V_i$. We have  $|G|=2|N_1|s$. The elements of $H$ with eigenvalue $1$ are those of $N_1\cup N_2$,
	so
	\begin{equation}
		\label{eq:impr_precise}
		\alpha(G)\ge \frac{2|N_1|-1}{2|N_1|s} \ge \frac{1}{2s} \ge \frac{1}{2(q^{d/2}-1)}.
	\end{equation}
	Moreover, $\delta(V\rtimes G)\ge 1/(4(q^{d/2}-1))>f(q^d)$ by \Cref{lemma_delta_alpha}, so the case where $G$ is semiregular on $V\sm\{0\}$ is done. Assume then $G$ is not semiregular on $V\sm\{0\}$. If $|N_1|\ge 2$, then we can improve the second inequality in \eqref{eq:impr_precise} and get $\alpha(G)\ge 3/(4(q^{d/2}-1)) > h(q^d)$. What is more, the nontrivial elements of $N_1\cup N_2$ fix $q^{d/2}$ vectors, so an easy modification of the proof of \Cref{lemma_delta_alpha} gives $\delta(V\rtimes G)\ge 3/(4(q^{d/2}-1))\cdot(1-1/q^{d/2})> g(q^d)$.
 Note also that $N_1\times N_2 \le A(G)$, so  $|A(G)|/|G| \ge |N_1||N_2|/(2|N_1|s) \ge 1/(q^{d/2}-1)$,  and so the case where $|N_1|\ge 2$ is  done. Assume finally $N_1=1$, so $|G|=2s$. If $G$ contains at least three elements with eigenvalue $1$, then
	\[
	 \alpha(G)\ge \frac{3}{|G|} = \frac{3}{2s} \ge \frac{3}{2(q^{d/2}-1)},
	\]
and so by \Cref{lemma_delta_alpha} $\delta(V\rtimes G)\ge 3/(4(q^{d/2}-1))$ and we are done. Assume then $G$ contains a unique nontrivial element $x$ with eigenvalue $1$; then $|x|=2$, and $\alpha(G) = 2/|G| \ge 1/(q^{d/2}-1)$. We have $x\in G\sm H$,
so write $x=(x_1,x_2)\tau$ where $1\neq \tau \in S_2$ and $x_i\in \GL_{d/2}(q)$. Since $|x|=2$, we have $x_1x_2=1$, from which $x$ fixes $q^{d/2}$ vectors and so as above $\delta(V\rtimes G)\ge 1/(q^{d/2}-1)\cdot(1-1/q^{d/2})>g(q^d)$ and we are done.
	\end{proof}

We now prove \Cref{t:affine} assuming \Cref{prop:GamL_1,prop:GamL_2,prop:combination}.

\begin{proof}[Proof of \Cref{t:affine} assuming \Cref{prop:GamL_1,prop:GamL_2,prop:combination}]
Let $G\le \GL_d(q)$ be irreducible. In view of \Cref{l:general_imprimit}, we may assume  $G$ is primitive. From now, for convenience we change notation and replace $d/r$ by $d$ and $q^r$ by $q$. In particular, $G\le  \GL_{dr}(q^{1/r})$ is primitive and $d$ is minimal so that $G\le \GamL_d(q)$.

Let $H:=G\cap \GL_d(q)$. By Clifford's theorem, each characteristic subgroup $L$ of $H$ acts homogeneously. Let $V_1$ be an $\F_q L$-component, whose dimension we denote by $a$. By Schur's lemma, $\End_{\F_q L}(V_1)$ is a field extension of $\F_q$, say of degree $b$. By \cite[3.11]{Asc}, $\C_{\GL_d(q)}(L)\cong \GL_{d/b}(q^b)$, and since $G$ normalizes $L$, we get $G\le \GamL_{d/b}(q^b)$. By the minimality of $d$, we deduce $b=1$, which is to say, $V_1$ is absolutely irreducible.

Choose now $L=\mathrm F^*(H)$. We have $L=\Z(H) R$, where $R= G_1 \circ \cdots \circ G_t$ and the following holds. 
There exists $0\le r \le t$ such that for $1\le i\le r$, $G_i$ is an extraspecial $p_i$-group of order $p_i^{2r_i+1}$, of exponent $p_i$ if $p_i$ is odd; and possibly replacing $G_i$ by $ZG_i$ where $Z\le \F_q^\times$ has order $4$,
we have $H/\C_H(G_i)\le  p_i^{2r_i}.\Sp_{2r_i}(p_i)$. 
For $r+1\le i\le t$, $G_i$ is a central product of $\ell_i$ copies of a quasisimple group $S_i$, and $H/\C_H(G_i) \le \Aut(G_i)=\text{Aut}(S_i)\wr S_{\ell_i}$.

Recall now that $\Z(L)=\C_H(L)$. We have
\[
H/\C_H(L)=H/\bigcap_i \C_H(G_i) \le \prod_{i=1}^t H/\C_H(G_i)
\]
and therefore
\begin{equation}
	\label{eq:index}
\frac{|H|}{|L|} \le \prod_{i=1}^r |O_i|
\end{equation}
where for $i\le r$, $O_i := \Sp_{2r_i}(p_i)$, and for $i\ge r+1$, $O_i:=
\text{Out}(S_i)\wr S_{\ell_i}$.

\textbf{Assume first} that $L$ acts (absolutely) irreducibly on $V$. If $d\le 2$, we have $G\le \GamL_1(q)$ or $\GamL_2(q)$ and we conclude by \Cref{prop:GamL_1,prop:GamL_2}. Assume then $d\ge 3$. We will prove that 
\begin{equation}
	\label{eq:desired_inequality}
\alpha(L,V)> 2\log(q)|H:L|q^{-d/2},
\end{equation}
which implies $|A(G,V)|/|G|\ge \alpha(G, V) > 2/q^{d/2}$  since $|G:L|\le \log(q)|H:L|$, and so $\delta(V\rtimes G,V)> 1/q^{d/2}$ by \Cref{lemma_delta_alpha}. This will conclude the proof.

Since $L$ is absolutely irreducible and $L=\Z(H)R$,  $R$ is also absolutely irreducible. Then, by \cite[(3.16)]{Asc} we deduce that $W=W_1\otimes \cdots \otimes W_t$, where $W_i$ is an absolutely irreducible $\F_qG_i$-module of dimension $d_i$. In particular, for $1 \le i \le r$, $W_i$ is faithful and $d_i=p_i^{r_i}$. For $r+1\le i\le t$, $W_i = M_{i_1}\otimes \cdots \otimes M_{i\ell_i}$ where $M_{i_j}\cong M_i$ is a faithful absolutely irreducible $\F_qS_i$-module of dimension $m_i$ (again by \cite[(3.16)]{Asc}). (In other words, $W_i$ is a tensor product of $\ell_i$ copies of $M_i$, but it will sometimes be useful to make use of the indices $i_j$.) We may take it that if $S_i\cong S_j$ with $i\neq j$ then $M_i\not\cong M_j$.

Let now $C>0$ satisfy the following conditions:
\begin{itemize}
    \item[$\diamond$] $C$ satisfies the conclusion of \Cref{prop:combination}.
    \item[$\diamond$] If $q$ is a power of $s$, $s(s^2-1)\ge C$, and either $\ell\ge 3$ or $\ell \ge 1$ and $s<q$, then 
\begin{equation}
\label{eq:SL2}
    2\log(q)(2s\log(s))^\ell  \ell! < q^{2^{\ell-1}}.
\end{equation}.
\item[$\diamond$] If $q(q^2-1)\ge C$ and $a\ge 2$ then
\begin{equation}
\label{eq:SL2newcondition}
2\log(q)(24q\log^2(q))^{\floor{a/2}}(1+\delta\cdot q\log(q)) < q^{2^{a-1}},
\end{equation}
where $\delta=0$ if $a=2$, and $\delta=1$ if $a\ge 3$.
\end{itemize}
Let us also record the following inequality, which holds for $d\ge 3$, $\ell\ge 1$, and $q$ a prime power:
\begin{equation}
\label{eq:new_requirement}
\ell! q^{d\ell/2} \le q^{d^\ell/2}.
\end{equation}

Now that we fixed $C$, let $g(C)$ be so that for every $\ell\ge g(C)$ and every $m\ge 2$, 
\begin{equation}
\label{eq:g(C)}
 2C^{C\ell} \ell!q^{m\ell} \log(q) < q^{m^\ell/2}
\end{equation}
(It is easily seen that $g(C)$ exists.) For later use, note that if $G$ is a finite group of order at most $C$, then (crudely) $|\Out(G)|\le C^C$.

Let now $f(C)$ be equal to $1$, plus the product of all the following quantities: For every extraspecial group of order $s^{2f+1}\le C$, $s^{2f+1}|\Sp_{2f}(s)|$; For every quasisimple group $L$ of order at most $C$ and for every $\ell \le g(C)$, $(|\Aut(L)|^\ell \ell!)^{|L|}$. (For a later use, note that $|L|$ is an upper bound for the number of equivalence classes of absolutely irreducible $\F_q L$-representations.)

Now let
\begin{itemize}
    \item[$\diamond$] $\Omega_1$ be the set of indices $1\le i \le t$ such that either $i\le r$ and $|G_i|\le C$, or $i\ge r+1$ and $|S_i|\le C$ and $\ell_i \le g(C)$. 

\item[$\diamond$] $\Omega_2$ be the set of indices $r+1\le i \le t$ such that $|S_i|\le C$ and $\ell_i > g(C)$. 

\item[$\diamond$] $\Omega_3$ be the set of indices $r+1\le i \le t$ not belonging to $\cup_{j\le 2}\Omega_j$ and such that $m_i=\dim(M_i)=2$, $\ell_i\le 2$ and $S_i\cong \SL_2(q)$.
\item[$\diamond$]  $\Omega_4$ be the set of indices $r+1\le i \le t$ not belonging to $\cup_{j\le 2}\Omega_j$ and such that $m_i=\dim(M_i)=2$, and $\ell_i \ge 3$ or $S_i\cong \SL_2(q_i)$ with $q_i<q$.

\item[$\diamond$] $\Omega_5$ be the set of remaining indices. 
\end{itemize}

For $j=1, \ldots, 5$, denote $d(\alpha_j):=\prod_{i\in \Omega_i}d_i$.  Since $n$ is sufficiently large, we may assume that $\cup_{i>1}\Omega_i \neq \varnothing$. Now choose any $j\in \cup_{i>1}\Omega_i$. If $j\le r$ replace $G_j$ by $\Z(H)G_j$; if $j\ge r+1$ replace $S_{j_1}$ by $\Z(H)S_{j_1}$. From now on we will not specify this in the notation. In particular, $L=G_1\circ \cdots \circ G_t$.

We make use of the following easy observation: If $g_i\in G_i$ has eigenvalue $1$ on $W_i$, then the image of $(g_1, \ldots,g_t)$ in $L$ has eigenvalue $1$ on $V$. We deduce that
\begin{equation}
\label{eq:total}
\alpha(L, V)\ge \prod_{1\le i \le 5} \alpha_i
\end{equation}
where $\alpha_i = \prod_{j\in \Omega_i} \alpha(G_j, W_j)$.

Observe that by definition of $f(C)$,
\begin{equation}
\label{eq:alpha1}
\frac{\alpha_1}{\prod_{i\in \Omega_1}|O_i|} \ge \frac{1}{\prod_{i\in \Omega_1}|G_i||O_i|} > \frac{1}{f(C)}.
\end{equation}
Moreover, by \cite{cameron_cohen} and Lemma \ref{lemma_delta_alpha}, for $i\in \Omega_2$, $\alpha(S_i,M_i)\ge \delta(M_i\rtimes S_i,M_i)\ge 1/q^{m_i}$, therefore by \Cref{eq:g(C)}, noting that $d_i=m_i^{t_i}$, we get that if $\Omega_2\neq \varnothing$ then
\begin{equation}
\label{eq:alpha2}
\frac{\alpha_2}{\prod_{i\in \Omega_2}|O_i|} \ge \prod_{i\in \Omega_2} \frac{1}{C^{C\ell_i}\ell_i! q^{m_i\ell_i}} > 2\log(q) \prod_{i\in \Omega_2} \frac{1}{q^{d_i/2}} \ge \frac{2\log(q)}{q^{d(\alpha_2)/2}}.
\end{equation}
(Note that the assumption that $\Omega_2\neq \varnothing$ is only needed to allow the factor $2\log(q)$. The same will hold for $\Omega_4$ and $\Omega_5$, below.)

By \Cref{prop:combination}, by the definition $C$ and by \eqref{eq:new_requirement}, we have that if $\Omega_5\neq \varnothing$ then
 \begin{equation}
\label{eq:alpha5}
\frac{\alpha_5}{\prod_{i\in \Omega_5}|O_i|}  > 2\log(q) \prod_{\substack{1\le i\le r \\ i\in \Omega_5}} \frac{1}{q^{d_i/2}}  \prod_{\substack{r+1\le i\le t \\ i\in \Omega_5}} \frac{1}{\ell_i! q^{m_i\ell_i/2}} \ge \frac{2\log(q)}{q^{d(\alpha_5)/2}}.
\end{equation}

Note now that for $i\in \Omega_3\cup\Omega_4$, $S_i=\SL_2(q_i)$ or $\Z(H)\SL_2(q_i)$, where $q$ is a power of $q_i$ and $M_i$ is either the natural module or the dual. (And there is at most one $i_j$ for which $S_{i_j}\neq \SL_2(q_i)$.) We have 
$\alpha(S_i,M_i)\ge 1/q_i$ (see \Cref{l:simple_estimate}).
Using $|O_i|\le 2\log(q_i)$ and
 \eqref{eq:SL2} (applied with $s=q_i$) we deduce that if $\Omega_4\neq \varnothing$ then
\begin{equation}
\label{eq:alpha4}
\frac{\alpha_4}{\prod_{i\in \Omega_4}|O_i|} \ge   \prod_{i\in \Omega_4} \frac{1}{(2\log(q_i)q_i)^{\ell_i}\ell_i!}  > \frac{2\log(q)}{q^{d(\alpha_4)/2}}
\end{equation}

Now, let us address $\Omega_3$. If $a:=\log(d(\alpha_3)) = \sum_{i\in \Omega_3}\ell_i$ is even, we choose an arbitrary matching of the factors $S_{i_j}$,
 $i\in \Omega_3$, $1\le j\le \ell_i$.  If $a$ is odd, we do the same by leaving out an arbitrary factor. If $S_\nu$ and $S_\eta$ are matched, we readily see that $\alpha(S_\nu\circ S_\eta,M_\nu\otimes M_\eta)\ge 1/(3q)$.
(Indeed, choose an element of $S_\nu$ having eigenvalues in $\F_q^\times$ -- for $q\ge 7$ the proportion of choices is at least $(q-3)/(2(q-1))\ge 1/3$ --
and choose an element of $S_\eta$ having eigenvalue $\lambda^{-1}$, where $\lambda$ is an eigenvalue of the first element -- the proportion of choices is at least $1/q$, see \Cref{l:simple_estimate}.) Using that $|O_i|\le 2\log(q)$, we get
\begin{equation}
\label{eq:alpha3}
\frac{\alpha_3}{\prod_{i\in \Omega_3}|O_i|} \ge \frac{1}{q\log(q)} \frac{1}{(24q\log^2(q))^{\floor{a/2}}} 
\end{equation}
where the first factor $1/(q\log(q))$ appears only if $a$ is odd.

Is is now easy to deduce the desired conclusion \eqref{eq:desired_inequality}, using the assumption $d\ge 3$. Indeed, if $a\ge 2$ then by \Cref{eq:SL2newcondition}, the right-hand side of \eqref{eq:alpha3} is at least $2\log(q)/q^{d(\alpha_3)/2}$, and for $a=1$ it is equal to $1/(q\log(q))$. In particular, if $a=0$ or $a\ge 2$ then we get from \Cref{eq:index,eq:alpha1,eq:alpha2,eq:alpha3,eq:alpha4,eq:alpha5}
\[
\frac{\alpha(L,V)}{|H:L|}> \frac{2\log(q)}{f(C)q^{(d(\alpha_2)+d(\alpha_3)+ d(\alpha_4)+ d(\alpha_5))/2}}
\]
which for sufficiently large $qd$ is at least $2\log(q)q^{-d/2}$, giving \eqref{eq:desired_inequality} as desired. If $a=1$, then the assumption $d\ge 3$ implies that $\Omega_1\cup \Omega_2\cup \Omega_4\cup \Omega_5\neq \varnothing$, and one concludes similarly, using that $d(\alpha_i)=1$ or $\ge 3$ for $i=2,4,5$.

\textbf{Assume now} $L$ is not irreducible. Letting $V'$ be a component, we have $\alpha(L,V')=\alpha(L,V)$ and $|V'|\le |V|^{1/2}$, so if $d':=\dim(V')\ge 3$ then the conclusion follows immediately from \Cref{eq:desired_inequality} (and in fact, we can simply use $|V'|\le |V|$). If $d'=2$,
then $d\ge 4$ and  by \Cref{prop:GamL_2}, $\alpha(L,V')\gg 1/q$. Since $|H:L|\ll 1$
we have $\alpha(L,V)\gg 1/q \ge 2\log(q) |H:L|q^{-d/2}$ for $qd$ large.

This concludes the proof of \Cref{t:affine} assuming \Cref{prop:GamL_1,prop:GamL_2,prop:combination}.
\end{proof}

\section{$\GamL_1(q)$}
\label{sec:GammaL1}

In this section we prove \Cref{t:affine} in the case where $d=r$ (that is, \Cref{prop:GamL_1}). The bulk of the proof boils down to elementary number theory.

We will work with any subgroup of $\GamL_1(q^d)$ (with no irreducibility assumption); for convenience, we replace $q^d$ by $q$ in the notation, so $G\le \GamL_1(q)$ is not semiregular on $V\sm\{0\}$. 

Let us now fix some notation. Write $q=p^f$, where $G\le \GL_1(q)\rtimes \mathrm{Gal}(\F_q/\F_p)$ and $G$ projects onto $\mathrm{Gal}(\F_q/\F_p)$ (so in this section $p$ need not be prime). Since $G$ is not semiregular on $V\sm\{0\}$, we have $f>1$.

Denote $H:=G\cap \GL_1(q)$ and $t:=|H|$, and write $G=\gen{H,z}$, with $z=\tau x$, $\tau$ is the $p$-th power map and $x\in \GL_1(q)$. Let $\overline x:=xH\in \GL_1(q)/H$ and $m:=|\overline x|$, so $m \mid (q-1)/(p-1)$ and $m\mid (q-1)/t$. We denote by $\Delta$ the set of nontrivial elements of $G$ fixing a nonzero vector, and by $(a,b)$ the greatest common divisor of the positive integers $a$ and $b$.

Fix a coset $z^\ell H$, with $\ell$ a proper divisor of $f$.

\begin{lemma}
\label{l:agl_1_coset}
We have that $\Delta\cap z^\ell H\neq \varnothing$ if and only if $\overline x^{(p^\ell-1)/(p-1)}$ is a $(p^\ell-1)$-th power. Moreover, if this is the case, 
\[
|\Delta\cap z^\ell H|=\Big(\frac{q-1}{p^\ell-1}, t\Big). 
\]
\end{lemma}

\begin{proof}
Write $y:=x^{(p^\ell-1)/(p-1)}$, so $z^\ell=\tau^\ell y$. For $0\neq v\in \F_q$ and $h\in H$ we have $vz^\ell h = v^{p^\ell}yh$, and this is equal to $v$ if and only if $yh = v^{1-p^\ell}$. In particular, there exist $h$ and $v\neq 0$ such that $vz^\ell h=v$ if and only if $y$ is a $(p^\ell-1)$-th power modulo $H$.

Assume now that $yh=v^{1-p^\ell}$. Then $\Delta\cap z^\ell H$ is the set of elements obtained by multiplying $z^\ell h$ by a $(p^\ell-1)$-th power of $\GL_1(q)$ contained in $H$. The number of these elements is $((q-1)/(p^\ell-1),t)$, and this concludes the proof.
\end{proof}

We introduce further notation. For a prime number $r$ and a positive integer $a$, we let $\gamma_r(a)$ be the integer $i$ such that $r^i$ is the $r$-part of $a$. Moreover, we let $\overline\gamma_r(a):=\gamma_r((p^a-1)/(p-1))$. Finally, we write $m=(q-1)/(tC)$ for a positive integer $C$.

We will often be concerned with several primes $r_1, \ldots, r_b$. In this case, for convenience we will write $\gamma^i(a)=\gamma_{r_i}(a)$, and similarly $\overline\gamma^i(a)=\overline\gamma_{r_i}(a)$.
\begin{lemma}
\label{l_r_part_smaller}
Let $r_1, \ldots, r_b$ be the distinct prime divisors of $(f,p-1)$ (possibly $b=0$), and let $\ell<f$ be a divisor of $f$. Then, $\Delta\cap z^\ell H\neq \varnothing $ if and only if, for every $i=1, \ldots, b$, one of the following holds:
\begin{itemize}
    \item[(i)] $\gamma^i(m)\le \overline\gamma^i(\ell)$.
    \item[(ii)] $\gamma^i(p^\ell-1) \le \gamma^i(C) + \overline\gamma^i(\ell)$, i.e., $\gamma^i(p-1) \le \gamma^i(C)$.
\end{itemize}
\end{lemma}

\begin{proof}
By \Cref{l:agl_1_coset}, $\Delta\cap z^\ell H\neq \varnothing$ if and only if $\overline x^{(p^\ell-1)/(p-1)}$ is a $(p^\ell-1)$-th power. Note that $\overline x^{(p^\ell - 1)/(p-1)}$ is a $(p^\ell -1)$-th power if and only if
\[
\frac{m}{(m,(p^\ell-1)/(p-1))} \mid \frac{(p^f-1)/t}{((p^f-1)/t, p^\ell-1)}.
\]
Using that $Cm=(p^f-1)/t$, we see that this is equivalent to 
\[
(Cm, p^\ell-1) \mid C\cdot (m,(p^\ell-1)/(p-1)).
\]
This is equivalent to saying that, for every prime number $r$,
\begin{equation}
\label{eq_cond}
\min\{\gamma_r(C) + \gamma_r(m),\gamma_r(p^\ell-1)\} \le \gamma_r(C) + \min \{\gamma_r(m), \overline\gamma_r(\ell)\}.
\end{equation}
We need to show that \eqref{eq_cond} holds for every $r$ if and only, for every $i=1,\ldots, b$, (i) or (ii) in the statement holds.

We first claim that \eqref{eq_cond} holds for $r=r_i$ if and only if (i) or (ii) holds for $i$. This is a straightforward check; let us show, for instance, that (i) or (ii) in the statement for $i$ implies \eqref{eq_cond} for $r=r_i$. Assume (i). Then the RHS of \eqref{eq_cond} for $r=r_i$ is $\gamma^i(C) + \gamma^i(m)$. But this term appears in the minimum of the LHS, so \eqref{eq_cond} holds. Assume, now, (ii), and assume that (i) does not hold (otherwise we are in the previous case). Then $\gamma^i(m)> \overline\gamma^i(\ell)$, so by (ii) the LHS is $\gamma^i(p^\ell-1)$; moreover the RHS is $\gamma^i(C) + \overline\gamma^i(\ell)$, hence \eqref{eq_cond} holds by (ii).

The converse implication for $r=r_i$ is proved similarly: Assume that \eqref{eq_cond} holds for $r=r_i$, and assume that (i) does not hold for $i$; then deduce easily that (ii) must hold.

In order to conclude the proof, it is sufficient to show that \eqref{eq_cond} always holds for every prime $r$ not dividing $(f,p-1)$. In order to check this, we may assume $\gamma_r(m)\ge 1$, otherwise \eqref{eq_cond} holds easily. So assume that $\gamma_r(m)\ge 1$ and that $r$ does not divide $(f,p-1)$.

If $r$ does not divide $p-1$, then $\gamma_r(p^\ell-1) = \overline\gamma_r(\ell)$. Now, if $\gamma_r(m)\le \overline\gamma_r(\ell)$, then by the same argument as above \eqref{eq_cond} holds. If, instead, $\gamma_r(m)\ge \overline\gamma_r(\ell)$, then the RHS is $\gamma_r(C) + \overline\gamma_r(\ell) = \gamma_r(C) + \gamma_r(p^\ell-1)$, and $\gamma_r(p^\ell-1)$ appears in the minimum in the LHS, hence \eqref{eq_cond} holds.

Therefore, in order to conclude the proof, it is sufficient to show that, if $\gamma_r(m)\ge 1$ and $r$ does not divide $(f,p-1)$, then $r$ does not divide $p-1$. Assume the contrary. We have $r\mid m\mid (p^f-1)/(p-1)$, hence $r\mid (p^f-1)/(p-1) - (p-1) = 2+p^2+\cdots + p^{f-1}$. But $r$ divides also $p^2-1$; hence $r$ divides $2+p^2+\cdots + p^{f-1}- (p^2-1) = 3+p^3+\cdots +p^{f-1}$. Going on in this way, we see that $r$ divides $f$, hence it divides $(f,p-1)$, which is a contradiction. This concludes the proof.
\end{proof}

The previous lemma reduces the problem to elementary number theory.
We need a few more lemmas. 

\begin{lemma}
\label{l_coprime_no_growth}
Let $r$ be a prime, and assume $\gamma_r(p-1)\ge 1$ and  $r \nmid j$. Then $\gamma_r(p^j-1)=\gamma_r(p-1)$.
\end{lemma}
\begin{proof}
Note that $(p^j-1)/(p-1) = 1+p+\cdots + p^{j-1} \equiv j \not\equiv 0\pmod r$.
\end{proof}

\begin{lemma}
\label{l:numbers}
Let $r$ be a prime and assume $i\ge 1$ and $(r,i)\neq (2,1)$. If $\gamma_r(p-1)=i$, then $\gamma_r(p^r-1)=i+1$.
\end{lemma}

\begin{proof}
Assume $p-1=ar^i$, with $(a,r)=1$, so $p\equiv ar^i +1 \pmod{r^{i+2}}$. Then $p^r \equiv (ar^i+1)^r \equiv a^rr^{ir} + rar^i + 1 \pmod{r^{i+2}}$. If $(r,i)\neq (2,1)$, then this is $\equiv ar^{i+1} + 1\pmod{r^{i+2}}$, whence the statement holds.
\end{proof}

Note that for $(r,i)= (2,1)$ \Cref{l:numbers} does not hold. Indeed, if $\gamma_2(p-1)=1$, then $\gamma_2(p^2-1)\ge 3$.

\begin{lemma}
\label{l_r_odd_grows_by_one}
Let $r$ be an odd prime, let $i\ge 1$ and assume $\gamma_r(p-1)\ge 1$. Then $\gamma_r((p^{r^i}-1)/(p-1))=i$. In particular, $\gamma_r((q-1)/(p-1))=\gamma_r(f)$.
\end{lemma}

\begin{proof}
By \Cref{l:numbers}, $\gamma_r(p^r-1) = \gamma_r(p-1)+1$, hence the first assertion of the statement holds for $i=1$; then use induction. The last assertion (``In particular...'') follows from this and Lemma \ref{l_coprime_no_growth}.
\end{proof}
The case $r=2$ is different.

\begin{lemma}
\label{l_r_even_growth}
Let $i\ge 1$ and assume $\gamma_2(p+1) = j$, with $j\ge 1$. Then $\gamma_2((p^{2^i}-1)/(p-1))=i+j-1$. In particular, $\gamma_2((q-1)/(p-1))=\gamma_2(f)+j-1$.
\end{lemma}

\begin{proof}
By induction on $i$. The case $i=1$ is the hypothesis. For $i\ge 2$ write $p^{2^i} - 1 = (p^{2^{i-1}} +1)(p^{2^{i-1}} -1)$. We have $p^{2^{i-1}}+1 \equiv 2 \pmod 4$ (since $2^{i-1}$ is even), hence $\gamma_2(p^{2^i}-1)= \gamma_2(p^{2^{i-1}}-1) + 1$, whence  the statement follows by induction and \Cref{l_coprime_no_growth}.
\end{proof}

We are now ready to prove \Cref{prop:GamL_1}.

\begin{proof}[Proof of \Cref{prop:GamL_1}]  We need to show that if $G$ is not semiregular on $V\sm\{0\}$, then $\alpha(G)\ge g(q)$, $|A(G)|/|G|\ge 1/(q^{1/2}-1)$, and $\delta(V\rtimes G)\ge h(q)$,  where  $g$ and $h$ are defined in \eqref{eq:def_g} and \eqref{eq:def_h}.
	
Assume first there exists $\ell \mid f$, $\ell < f/2$, such that $\Delta\cap z^\ell H\neq \varnothing$.  By \Cref{l:agl_1_coset},
    \[
\alpha(G)> \frac{((q-1)/(p^\ell-1),t)}{tf} \ge \frac{1}{(p^\ell-1,t)f} > \frac{1}{fp^\ell}.
\]
(In the second inequality we used that $(a,bc) \le (a,b)(a,c)$.) 
For $\ell\le f/3$ we have $fp^\ell \ll q^{1/3}$, hence this case is done.

Assume now that $f$ is even and $\Delta\subseteq z^{f/2}H$. We will in fact prove the desired bound for every $q$. We start with the bound to $\alpha(G)$; the bounds to $\delta(V\rtimes G)$ and $|A(G)|/|G|$ will follow easily at the end.

Write $(t,p^{f/2}-1) = (p^{f/2}-1)/B$. Then
\begin{equation}
\label{eq:divisions}
m\mid \frac{p^f-1}{t} \mid B(p^{f/2}+1).    
\end{equation}
Let $r_1, \ldots, r_b$ denote the distinct prime divisors of $(f,p-1)$ (possibly $b=0$). By Lemma \ref{l_r_part_smaller}, for every $i=1,\ldots, b$, one of the following holds:
\begin{enumerate}
    \item $\gamma^i(m) \le \overline\gamma^i(f/2)$.
    \item $\gamma^i(p-1) \le \gamma^i(C)$.
\end{enumerate}

Let $s$ be a prime divisor or $f$. Note that, for every $i$ such that $s\neq r_i$, by Lemma \ref{l_coprime_no_growth} we have $\overline\gamma^i(f/s)=\overline\gamma^i(f)$. In particular, since $m\mid(p^f-1)/(p-1)$,  certainly $\gamma^i(m) \le \overline\gamma^i(f) = \overline\gamma^i(f/s)$. If $s$ is odd then by assumption $\Delta\cap z^{f/s}H=\varnothing$. We deduce from Lemma \ref{l_r_part_smaller} that
\begin{itemize}
    \item[(a)] For any odd prime $s$ dividing $f$, we have that $s\mid(p-1)$.
\end{itemize}

Let now $i\in \{1, \ldots, s\}$ be such that $r_i$ is odd (if it exists). By the same reasoning as above (with $s=r_i$), since $\Delta\cap z^{f/r_i}H=\varnothing$, by Lemma \ref{l_r_part_smaller} we must have $\gamma^i(m) > \overline\gamma^i(f/r_i)$ and $\gamma^i(p-1) > \gamma^i(C)$.

Moreover, by Lemma \ref{l_r_odd_grows_by_one}, $\overline\gamma^i(f/r_i) = \overline\gamma^i(f)-1$. Recall, now, that (1) or (2) holds. But (2) cannot hold, because it would contradict the final inequality of the previous paragraph; so (1) holds. We have, therefore, $\overline\gamma^i(f/r_i) < \gamma^i(m) \le \overline\gamma^i(f/2) = \overline\gamma^i(f/r_i) + 1 = \overline\gamma^i(f)=\gamma^i(f)$ (the last equality follows from  \Cref{l_r_odd_grows_by_one}), so it must be $\gamma^i(m) = \overline\gamma^i(f)=\gamma^i(f)$. 
We summarize this:
\begin{itemize}
\item[(b)] For any $i$ such that $r_i$ is odd, $\gamma^i(m) =\overline\gamma^i(f)= \gamma^i(f)$.
\end{itemize}

Then, we claim that

\begin{itemize}
    \item[(c)] Either $p$ is odd, or $f\equiv 2 \pmod 4$.
\end{itemize}
Indeed, assume, by contradiction, that $p$ is even and $f\equiv 0 \pmod 4$. Then, for every $i$, $r_i$ is odd
and $\overline\gamma^i(f/2) = \overline\gamma^i(f/4)=\overline\gamma^i(f)$, hence $\Delta\cap z^{f/2}H\neq \varnothing$ if and only if $\Delta\cap z^{f/4}H\neq \varnothing$, which contradicts our assumption. This proves (c).

Now, by \Cref{l:agl_1_coset} we have
\begin{equation}
\label{eq:first_GL1}
\alpha(G) = \frac{(t,p^{f/2}+1)}{tf} + \frac{1}{tf} \ge \frac{1}{f(t,p^{f/2}-1)} +\frac{1}{tf}= \frac{B}{f(p^{f/2}-1)} +\frac{1}{tf}.
\end{equation}
By (a) and (b), we have $\frac{f}{2^{\gamma_2(f)}}\mid m$, so by \Cref{eq:divisions} 
\begin{equation}
\label{eq:no_name}
tf \mid  2^{\gamma_2(f)}(p^f-1).
\end{equation}
By \eqref{eq:divisions}, $\frac{f}{2^{\gamma_2(f)}} \mid m\mid B(p^{f/2}+1)$, and  for $i$ such that $r_i$ is odd, $r_i \nmid p^{f/2}+1$, so that 
\begin{equation}
\label{eq:lasttt}
f\mid 2^{\gamma_2(f)}m \mid 2^{\gamma_2(f)}B.
\end{equation}.
If $\gamma_2(f)=1$, then by \eqref{eq:first_GL1}, \eqref{eq:no_name} and \eqref{eq:last_GL1} we get 
\begin{equation}
\label{eq:added1}
\alpha(G)\ge \frac{B}{f(p^{f/2}-1)} +\frac{1}{tf}\ge \frac{1}{2(p^{f/2}-1)} + \frac{1}{2(p^f-1)}= h(q),
\end{equation}
as wanted.

From now, therefore, we assume $N:=\gamma_2(f) \ge 2$. By (c), we have $p$ odd. If $t$ is odd, then $\gamma_2(B) = \gamma_2(p^{f/2}-1)$. 
By Lemma \ref{l_r_even_growth}, $\gamma_2(p^{f/2}-1)\ge N+j-1 \ge N$, where $j:=\gamma_2(p+1)\ge 1$.
Hence $\gamma_2(B)\ge \gamma_2(f)$, so by \eqref{eq:lasttt} $f\mid B$ and by \eqref{eq:first_GL1}
\[
\alpha(G) > \frac{B}{f(p^{f/2}-1)} \ge \frac{1}{p^{f/2}-1}
\]
and we are done. Assume now that $N\ge 2$ and $t$ is even. If $\gamma_2(t)=\gamma_2(p^f-1)$, then $m\mid \frac{p^f-1}{t}$ is odd, so $\gamma_2(m)=0\le \overline\gamma_2(f/4)$, therefore  $\Delta\cap z^{f/4}H\neq \varnothing$, contradiction. 

Assume finally that $N\ge 2$, $t$ is even and $\gamma_2(t)<\gamma_2(p^f-1)$. In particular, $\gamma_2(t) \le \max \{\gamma_2(p^{f/2}-1), \gamma_2(p^{f/2}+1)\}$. Also, clearly $\min \{\gamma_2(p^{f/2}-1), \gamma_2(p^{f/2}+1)\}=1$. It is easy, then, to deduce
\begin{equation}
   \label{eq:last_GL1} 
\alpha(G) = \frac{(t,p^{f/2}+1)}{tf} +\frac{1}{tf}= \frac{2}{f(t,p^{f/2}-1)} +\frac{1}{tf}= \frac{2B}{f(p^{f/2}-1)} +\frac{1}{tf}.
\end{equation}
Now, if $r_i$ is odd,  by the validity of (1) or (2) we deduce that either $\gamma^i(m)\le \overline\gamma^i(f/4)$ or $\gamma^i(p-1)\le \gamma^i(C)$. Since $\Delta\cap z^{f/4}H=\varnothing$, it follows from \Cref{l_r_part_smaller} that (2) does not hold,
and so $\gamma_2(m)\le \overline\gamma_2(f/2)$, and also $\gamma_2(m)> \overline\gamma_2(f/4)$. 
It follows  $\gamma_2(m) = \overline\gamma_2(f/2) = N+j-2\ge N-1$, so by \Cref{eq:lasttt} $f\mid 2m$, and by \eqref{eq:divisions} $ft\mid 2(p^f-1)$. Since by \eqref{eq:divisions} $m\mid B(p^{f/2}+1)$, we deduce $\gamma_2(B) \ge N+ j - 3 \ge N-2$, and so by \eqref{eq:lasttt} and \eqref{eq:last_GL1} we get
\begin{equation}
\label{eq:added_2}
\alpha(G) = \frac{2B}{f(p^{f/2}-1)}  + \frac{1}{tf}  \ge \frac{1}{2(p^{f/2}-1)} + \frac{1}{2(p^f-1)} = h(q).
\end{equation}
The lower bound to $\alpha(G)$ is proved. For the lower bound to $\delta(V\rtimes G)$, note that every element of $\Delta$ fixes exactly $p^{f/2}$ vectors,
and so by \Cref{derangements_eta} and some reorganization we get
\[
\delta(V\rtimes G) = \frac{q^{1/2}-1}{q}\left(\alpha(G) + \frac{1}{|G|q^{1/2}}\right).
\]
From inspection of this proof, we have either $|G|\le 2(q-1)$, or $\alpha(G)\ge 1/(q^{1/2}-1)$, in which case we are done. Therefore we may assume $|G|\le 2(q-1)$, and using $\alpha(G)\ge h(q)$, we get $\delta(V\rtimes G)\ge g(q)$ by some reorganization. Finally, for what concerns $|A(G)|/|G|$, by the proof of \Cref{l:agl_1_coset} we have $|A(G)|=2(t,p^{f/2}+1)$.
Then, the bound $|A(G)|/|G|\ge 1/(q^{1/2}-1)$ follows readily by inspection of the proof (specifically, as in \eqref{eq:first_GL1}  we get $|A(G)|/|G|\ge 2B/(f(p^{f/2}-1))$, and as in \eqref{eq:last_GL1} we get $|A(G)|/|G|= 4B/(f(p^{f/2}-1))$, and in both cases the given bounds for $B$, leading to \eqref{eq:added1} and \eqref{eq:added_2}, give the required estimate).
\end{proof}

\begin{remark}
\label{ex:bound_sharp}
Let $q=p^f$ with $f$ even and $G=\GL_1(q)\rtimes \gen{\sigma}$, where $\sigma$ is the $q^{1/2}$-th power map. By  \Cref{l:agl_1_coset}, the proportion of elements in $G\sm \GL_1(q)$ fixing a nonzero vector is 
\[
\frac{(q^{1/2}+1,t)}{2t} = \frac{1}{2(q^{1/2}-1)}
\]
(with $t=q-1$). Counting the identity, we have
\[
\alpha(G) = \frac{1}{2(q^{1/2}-1)} + \frac{1}{2(q-1)} = h(q)
\]
and so equality can be attained in \Cref{t:affine}(3). What is more, by the same calculation as at the end of the previous proof, we get $\delta(V\rtimes G)=g(q)$, and so equality can be attained in \Cref{t:main}(2).
\end{remark}

\section{$\GamL_2(q)$} 
\label{sec:GamL2}
In this section we prove \Cref{prop:GamL_2}. For convenience, we replace $q^{d/2}$ by $q$ in the notation.  The cases to consider are when $G$ normalizes $\SL_2(5)$, $Q_8$, or $\SL_2(q_0)$ with $q$ a power of $q_0\ge 4$. We consider each case in turn. In the first two cases, in some sense we will argue by reducing the problem to $\GammaL_1(q)$, and so will use several times \Cref{prop:GamL_1}, which we proved in \Cref{sec:GammaL1}.

\subsection{$\SL_2(5)$} Assume $q\equiv \pm 1 \pmod{10}$
and let $L:=\SL_2(5)$,  $L\le G \le \N_{\GamL_2(q)}(L)=LZ\rtimes \gen\tau$,
 where $Z=\F_q^\times$, $q=p^f$, $p$ prime, and $\tau$ is the $p$-th power map. Recall that $f(n)$ was defined in \eqref{eq:def_f}.
 
\begin{lemma}
\label{l:sl2(5)}
    For $q$ sufficiently large we have $\alpha(G) \ge 1/(60(q-1))$ and $\delta(V\rtimes G) \ge f(q^2)$. If moreover $G$ is not semiregular on $V\sm\{0\}$, then $\alpha(G)\ge 31/(60(q-1))$, $\delta(V\rtimes G) > 31/(60(q-1))\cdot (1-1/q)$, and $|A(G)|/|G| \ge 1/(q-1)$.
\end{lemma}

\begin{proof}
In this proof, whenever we write ``semiregular'', we mean ``semiregular on $V\sm\{0\}$''. Set $E:=G\cap Z \gen\tau$, and note that $|G:E|\le 60$.

Assume first that $G$ is semiregular. Then $E$ is semiregular and so $|E|\le q-1$, from which
\[
\alpha(G)\ge \frac{\alpha(E)}{60}\ge \frac{1}{60(q-1)}.
\]
Note also that the set of derangements of $G$ is $V\sm\{0\}$, and so
\[
\delta(V\rtimes G)=\frac{q^2-1}{|G|q^2}\ge \frac{q^2-1}{60(q-1)q^2} = f(q^2),
\]
as wanted.

Assume now that $G$ is not semiregular; we will show that $\alpha(G)\ge 31/(60(q-1))$, and we will deduce the bounds to $\delta(V\rtimes G)$ and $|A(G)|/|G|$ at the end. If $E$ is not semiregular, by \Cref{prop:GamL_1} we have $\alpha(G)\gg \alpha(E)\ge h(q) \gg 1/q^{1/2}$,
and the proof is complete (by \Cref{lemma_delta_alpha}). Assume then $E$ is semiregular, so $|E|\le q-1$. Assume first that $H:=G\cap \GL_2(q)$ is semiregular, and let $1\neq g=x\sigma\in G$ fix a nonzero vector, where $x\in \GL_2(q)$ and $\sigma$ is a power of $\tau$. Since $H$ is semiregular, $|g|=|\sigma|$, and so by \cite[Theorem 4.9.1]{gorenstein} there exists $y\in \GL_2(q)$ such that $g^y=z\sigma$, where $z\in Z$. Replacing $G$ by $G^y$, we are reduced to the case where $E$ is not semiregular, and so this case is done.

Assume finally that $E$ is semiregular and $H$ is not semiregular.  Note that in $L$ there are $30$ elements of order $4$, $20$ elements of order $3$, $20$ elements of order $6$, $24$ elements of order $5$, $24$ elements of order $10$, and $2$ central elements. 

Let $H=LZ'$ where $Z'\le Z$. If $3\nmid q$, then a nontrivial element of $H$ with eigenvalue $1$ is of the form $gz$, where $g\in L\sm Z$, $1\neq z\in Z'$, and $g$ has eigenvalue $z^{-1}$. It is easy to see that $\alpha(G)\ge 31/(60(q-1))$ (attained when $q\equiv 1$ mod $4$, $q\equiv -1$ mod $3$, $q\equiv -1$ mod $5$, and $Z'=Z$).
Assume finally that $3\mid q$, so elements of order $3$ have eigenvalue $1$. Since $q\equiv \pm 1 \pmod 5$, we also have  $q\equiv 1 \pmod 8$.
In particular, we may assume that $4\nmid |Z'|$, for otherwise we have $\alpha(G)\ge 51/(60(q-1))$. We can also assume that $|G|=60(q-1)$ and $|E|=q-1$, otherwise $\alpha(G)\ge 21/(30(q-1))$. Now, take $g\in L$ of order $4$ such that $\tau$ normalizes $\gen g$; up to conjugation we may take it that $J:=\gen{E,g}$ acts diagonally on $V$.  
Let $W$ be the space generated by the first basis vector, and let $\overline J$ be the group induced by $J$ on $W$; since $4\nmid |Z'|$, we see that $J\cong \overline J$.
If $\overline J$ is not semiregular on $W\sm\{0\}$ then by \Cref{prop:GamL_1} $\alpha(J)\ge \alpha(\overline J, W)\gg 1/q^{1/2}$ and so $\alpha(G)\gg \alpha(J)\gg 1/q^{1/2}$ and we are done.
 Therefore we assume that $\overline J$ is semiregular  on $W\sm\{0\}$, from which $|J|=|\overline J|\le q-1$ and so $|G|\le 30(q-1)$, contradicting our assumption. This proves the lower bound to $\alpha(G)$. For the lower bound to $\delta(V\rtimes G)$, note that every bound to $\alpha(G)$ in the proof is given by counting elements of $H=G\cap \GL_2(q)$ (or otherwise $\alpha(G)\gg 1/q^{1/2}$). In particular, every such nontrivial element fixes $q$ vectors, and so the bound follows easily from \Cref{derangements_eta}. Finally, for what concerns $A(G)$, we claim that $L\le A(H)$, and for this we may replace $H$ by its cover $\wt H=L\times Z'$.
 Then, $A(\wt H)$ is a normal subgroup of $\wt H$ projecting to a normal subgroup of $L$ not contained in $\Z(L)$, from which $L\le A(\wt H)$,
 as desired. Then $|A(G)|/|G|\ge |L|/|G|\ge 1/(q-1)$ and we are done.
\end{proof}

\begin{remark}
	\label{rem:SL_2(5)_frobenius}
	Assume $q\equiv -1\pmod{60}$ and let $G=ZL$ where $Z=\F_q^\times$. Then $G$ acts semiregularly on $V\sm\{0\}$ and so $\alpha(G)=1/(60(q-1))$ and $\delta(V\rtimes G)=(q+1)/(60q^2)=f(q^2)$. In particular, the bounds in \Cref{t:affine}(2) (and \Cref{t:main_2}(2)) are sharp.
\end{remark}

\subsection{$Q_8$} Let $q$ be odd and $Q_8\le G\le \N_{\GamL_2(q)}(Q_8)=KZ\rtimes \gen\tau$,
where $K$ is a double cover of $S_4$, $Z=\F_q^\times$, $q=p^f$, $p$ prime, and $\tau$ is the $p$-th power map. We set $\hat{S_4}:=\GL_2(3)$ and we denote by $\wt{S_4}$ the other double cover of $S_4$, so $K\in \{\hat{S_4}, \wt{S_4}\}$.

\begin{lemma}
\label{l:q8}
    For $q$ sufficiently large we have $\alpha(G) \ge 1/(24(q-1))$. If moreover $G$ is not semiregular on $V\sm\{0\}$, then $\alpha(G)\ge 13/(24(q-1))$, $\delta(V\rtimes G)\ge 13/(24(q-1))\cdot (1-1/q)$, and $|A(G)|/|G|\ge 1/(q-1)$.
\end{lemma}

\begin{proof}
As in the proof of \Cref{l:sl2(5)}, we set $E:=G\cap Z\gen\tau$ and when we write ``semiregular'' we mean ``semiregular on $V\sm\{0\}$''. 

Since $|G:E|\le 24$, if $G$ is semiregular we have
\[
\alpha(G)\ge \frac{\alpha(E)}{24}\ge \frac{1}{24(q-1)}.
\]
Assume then $G$ is not semiregular; we want to show $\alpha(G)\ge 13/(24(q-1))$, and we will deduce the bounds to $\delta(V\rtimes G)$ and $|A(G)|/|G|$ at the end of the proof. Setting $H:=G\cap \GL_2(q)$, exactly as in the proof of \Cref{l:sl2(5)}, we reduce to the case where $E$ is semiregular and $H$ is not semiregular. We assume (as we may) that if $p\equiv \pm 3 \pmod 8$ then $K=\hat{S_4}$.

Write $G=\gen{N,g_1z_1,g_2z_2\sigma}$, where $N:=G\cap K$, $g_1,g_2\in K$, $z_1,z_2\in Z$, $\sigma$ is a power of $\tau$, and $H=\gen{N,g_1z_1}$.
If $|N|=16$, or $48$, or $24$ with $K= \hat{S_4}$, define $\beta:=|N|$. If $|N|=8$,  define $\beta:=24$. Assume finally $|N|=24$ with $K=\wt{S_4}$. If $q\equiv -1\pmod 8$, or $q\equiv 1\pmod 8$ and $|z_1|_2=8$ and $g_1\not\in N$, define $\beta:=48$; in the other cases, define $\beta:=24$.

We will now show that either $\alpha(G)\gg 1/q^{1/2}$ (in which case the proof is complete) or $|G|\le \beta(q-1)/2$. Note that this is immediate if $\beta=48$, so for now we may assume $\beta<48$, and in particular $|N|<48$.

Let $J:=\gen{-1, g_1z_1,g_2z_2\sigma}$. If we show that $|J|\le q-1$, then $|G|\le |N|(q-1)/2$, which is enough in all cases since $\beta\ge |N|$. We will accomplish this stronger bound in some cases.

If we may take $g_1=g_2=1$ then $J\le E$ has order at most $q-1$ as desired, so assume this is not the case. Since $g_1$ and $g_2$ normalize $N$,
if $|N|=16$ then we may take $g_1=g_2=1$ and so these cases are excluded.
If $|N|=24$, then  we may take $g_1$ or $g_2$ is equal to $1$, and the other, call it $g$, is of $2$-power order. We may assume $g\not\in N$ as explained in this paragraph.

Since we are assuming $\beta<48$, we have either $K=\hat{S_4}$ or $q\equiv 1\pmod 8$. In particular, we may take $g$ diagonal, of order $2$ if $K=\hat{S_4}$, and of order $8$ if $K=\wt{S_4}$. Let $W$ be the space generated by the first basis vector, let $\overline J$ be the group induced by $J$ on $W$, and let $T$ be the kernel of $J\to \overline J$. Let also $Z':=\gen{-1,g_1z_1}$.
Assume first $g=g_2$. 
We have $|G:J| = |H:H\cap J| = 12/|H\cap J : Z'|$. 
Moreover, the natural map $T\to (H\cap J)/Z'$ is injective.
Now, if $\overline J$ does not act semiregularly on $W\sm\{0\}$, then by \Cref{prop:GamL_1} $\alpha(G)\gg \alpha(J)\ge \alpha(\overline J,W)/|T|\gg 1/q^{1/2}$ and the proof is concluded. Assume then $\overline J$ acts semiregularly; then $|\overline J|\le q-1$ and so $|J|=|\overline J||T|\le (q-1)|H\cap J : Z'|$, from which $|G|\le 12(q-1)$, as desired.  
Assume now $g=g_1$ and $K=\wt{S_4}$. Then $p\neq 3$ and since $\beta<48$, we have $q\equiv 1\pmod 8$ and $|z_1|_2\neq 8=|g|$.
We have $H\cap J \le \gen{-1,g_1z_1, z_1^2}$,
and the assumption $|z_1|_2\neq 8=|g|$ easily implies that the natural map $T \to (H\cap J)/Z'$ is injective.
We then conclude exactly as in the previous case.
Assume now $g=g_1$ and  $K=\hat{S_4}$, so $|g|=2$. Assume $g$ centralizes the space $W$ generated by the first basis vector. Then we see that $J\cong \overline J$,
and we conclude as in the  case $g=g_2$, dividing the case where $\overline J$ is semiregular on $W\sm \{0\}$ or not.

Assume finally $|N|=8$. If $g_1\in N$ or $g_2\in N$, then $|G|\le 12(q-1)$ as desired.
 If $g_1$ and $g_2$ are not in $N$, then we may take $|g_1|=3$ 
 and $g_2$ of $2$-power order. It follows that $K=\hat{S_4}$ or $q\equiv 1\pmod 8$. (Indeed, if $K=\wt{S_4}$ and $q\equiv -1\pmod 8$, we have $q=p^f$ with $f$ odd,
 in which case by choosing $g_2\in \GL_2(p)$, we have that $g_2z_2'$ is a power of $g_2z_2\sigma$ for some $z_2'\in Z$, which is impossible.)
In particular, we can pass to a subgroup $B$ of index at most $3$; the same argument as in the case $g=g_2$ in the previous paragraph then gives $|B| \le 4(q-1)$ and so $|G|\le 12(q-1)$, as desired.
Therefore we may assume $|G|\le \beta(q-1)/2$ in all cases, as claimed.

Recall we are assuming $H$ is not semiregular; let $1\neq y\in H$ with eigenvalue $1$. Assume $y=nxz$, where $n\in N$, $x\in K$, $z\in Z$, and $xz$ a power of $g_1z_1$ not belonging to $N\sm\{1\}$. 
Let $c:=nx$, so $y=cz$, and note that $|c|=2,3,4,6$ or $8$. If $|c|=2$ then $x=1$ and $y=n$.
If $|c|= 6$ or $8$ then we replace $y$ by $y^2$ and assume $|c|=3$ or $4$. Now note that, in all cases with $|c|=2,3$ or $4$, if $c'\in Nx$ is such that $|c'|=|c|$, then $c'z$ has eigenvalue $1$.
Moreover, if $|c|=3$ and $xz\neq 1$, then $c^2$ has eigenvalue $1$, and the number of elements of order $3$ in $Nxz$ is equal to the number of elements of order $3$ in $N(xz)^2$.
There is one case where we need to improve slightly the counting.
If $K=\wt{S_4}$, $|N|=24$, $|z_1|_2=8$ and $g_1\notin N$, then we can  take $|g_1|=8$, and up to replacing $z_1$ by $-z_1$, a power $g_1z_1$ of the form $g_1z_1'$ has eigenvalue $1$.
Furthermore, for every $\GL_2(q)$-conjugate $a$ of $g_1$ in $Ng_1$ (there are $6$ such conjugates),
$az'_1$ has eigenvalue $1$, and finally we find other $6$ nontrivial elements with eigenvalue $1$ in $N(g_1z'_1)^2 = N{z_1'}^2$. We counted $12$ elements, which is also the number of elements of order $4$ in $K\sm N$.

Let now $C$ be a coset of $N$ in $K$ normalizing $N$ and let $t\in \{3,4\}$. If $N=Q_8$, $C\subseteq P\sm N$ where $P\in \mathrm{Syl}_3(K)$, and $t=3$, then we let $F_C(t)=8$ (namely, the number of elements of $P$ of order $3$). 
If $p\neq 3$, $|N|=48$ or $|N|=24$, and $t=3$, then we let $F_C(t)=16$ (namely, twice the number of elements of $K$ of order $3$).  
In all other cases, we let $F_C(t)$ be the number of elements that are either in $C$ and have order $t$, or are in $N\sm \Z(N)$ and have order $2$. We finally let $f_C(t):=(F_C(t)+1)/\beta$.

It follows from the above discussion that there exists a coset $C$ of $N$ in $K$ normalizing $N$
and $t\in \{3,4\}$, such that if $|N|=24$, $K=\wt{S_4}$, $|z_1|_2=8$ and $g_1\not\in N$, then $C=K\sm N$ and $t=4$;
if $|N|=24$ and $q\equiv -1\pmod 8$ (so $p\neq 3$), then $t=3$; 
and such that moreover $f_C(t)>1/\beta$ and 
\[
\alpha(G)\ge \frac{F_C(t)}{|G|} + \frac{1}{|G|} \ge \frac{2f_C(t)}{(q-1)}.
\]
(In the second inequality we used $|G|\le \beta(q-1)/2$.)
It is straightforward to check that if $f_C(t)>1/\beta$ then $f_C(t)\ge 13/48$ (attained for example by $K=N=\hat{S_4}$, $q\equiv 3\pmod 8$ and $q\equiv 2\pmod 3$, and $t=3$ or $4$). 

This gives $\alpha(G)\ge 13/(24(q-1))$ as required. The lower bound to $\delta(V\rtimes G)$ follows as at the end of the proof of \Cref{l:sl2(5)},
and the lower bound to $|A(G)|/|G|$ follows by inspection of the proof.
\end{proof}

\subsection{$\SL_2(q_0)$}

We begin by recording  a simple estimate for the proportion of elements in a coset of $\SL_2(q)$ having a given eigenvalue in $\F_q$. 

\begin{lemma}
\label{l:simple_estimate}
    Let $g\in \GL_2(q)$ and let $\lambda \in \F_q^\times$. The proportion of elements of  $\SL_2(q)g$ having eigenvalue $\lambda$ is at least $1/q$.
\end{lemma}

\begin{proof}
    If $\lambda^2\neq \det(g)$, the proportion of (semisimple) elements of $\SL_2(q)g$ having eigenvalue $\lambda$ is $1/(q-1)$. If $\lambda^2= \det(g)$, the proportion of $\lambda$-potent elements is $1/q$.
\end{proof}

We turn then to the general case. Let $q$ be a power of $q_0$, with $q_0\ge 4$,  and $\SL_2(q_0) \le G \le \N_{\GamL_2(q)}(\SL_2(q_0))=\GL_2(q_0)Z\rtimes \gen\tau$,
where $Z=\F_q^\times$, $q=p^f$, $p$ prime, and $\tau$ is the $p$-th power map. This is the only case where the bound for $\delta(V\rtimes G)$ will require somewhat more work than the bound for $\alpha(G)$.

\begin{lemma}
\label{l:sl2(q_0)}
    We have $\alpha(G)> 3/(4q)$, $\delta(V\rtimes G)> 9/(16q)$, and $|A(G)|/|G|\ge 1/(q-1)$.
\end{lemma}

\begin{proof} 
	Let $H:=G\cap (\GL_2(q_0)\rtimes \gen\tau)$, so $|G:H| \le (q-1)/(q_0-1)$. 
	Let $W=\F_{q_0}^2$; we now count the elements of $H$ fixing a vector in $W\sm \{0\}$. Assume that $H$ projects onto $\Gal(\F_q/\F_r)$, let $N:=H\cap \GL_2(q_0)$, 
	and consider any coset $Ng$ with $g\in H$. Assume that $g$ projects onto $\Gal(\F_q/\F_{r^i})$, and set $s(i):=|\F_{q_0}\cap \F_{r^i}|$. By Shintani descent (see \cite[Theorem 2.14]{burness2013uniform})
	 and \Cref{l:simple_estimate}, the proportion of elements of $Ng$ that fix a nonzero vector of $W$ is at least $1/s(i)\ge 1/q_0$. Therefore, $\alpha(H,W)\ge 1/q_0$.
	Since $|G:H|\le (q-1)/(q_0-1)$, we have $\alpha(G,V) \ge (q_0-1)/(q_0(q-1)) \ge 3/(4(q-1))$ (since $q_0\ge 4$), as desired. Moreover, $A(H,W)$ is a normal subgroup of $H$ containing $N$
and	intersecting each $N$-coset in $H$, so $A(H,W)=H$
and $|A(G)|/|G| \ge |H|/|G| \ge (q_0-1)/(q-1)>1/(q-1)$.
	
	Next, we prove the lower bound to $\delta(V\rtimes G,V)$. Assume $q=r^m$. Note that in a coset $Ng$ where $g\in H$ projects onto $\Gal(\F_q/\F_{r^i})$, each element fixing a nonzero vector fixes at least $r^i$ vectors. Since in such coset we produced at least $|N|/s(i)$ elements fixing a nonzero vector, we deduce
	\[
	\eta(G,V)\le \frac{1}{|G|}\left(\sum_{i\mid m} \frac{\varphi(m/i)|N|}{s(i) r^i}+ |G|-\sum_{i\mid m} \frac{\varphi(m/i)|N|}{s(i)}\right),
	\]
		where $\varphi$ is Euler's totient function,
	and so by \Cref{derangements_eta}
	\begin{equation}
		\label{eq:SL2_field}
		\delta(V\rtimes G,V)\ge \frac{|N|}{|G|}\sum_{i\mid m} \frac{\varphi(m/i)(r^i-1)}{s(i)r^i} \ge \frac{q_0-1}{(q-1)m}\sum_{i\mid m} \frac{\varphi(m/i)(r^i-1)}{s(i)r^i},
	\end{equation}
	where in the last inequality we used $|G|\le (q-1)m|N|/(q_0-1)$. 	  Using $\sum_{i\mid m}\varphi(m/i)=m$, $(r^i-1)/r^i\ge (r-1)/r$, and $s(i)\le q_0$, we see that the right-hand side of \eqref{eq:SL2_field} is at least $(q_0-1)(r-1)/(q_0(q-1)r)$. In particular, for $r>2$ (and $q_0\ge 4$) this is at least $9/(16(q-1))$, which is enough.
	Assume then $r=2$. Using that $s(1)=2$, that for $i>1$ we have $(2^i-1)/2^i\ge 3/4$ and $s(i)\le q_0$,
	 we get
	\begin{align*}
	\frac{1}{m}\sum_{i\mid m} \frac{\varphi(m/i)(2^i-1)}{s(i)2^i} &\ge \frac{\varphi(m)}{4m} + \frac{3(m-\varphi(m))}{4mq_0} \\
 &= \frac{3}{4q_0}+ \frac{\varphi(m)}{m} \left(\frac{1}{4} - \frac{3}{4q_0}\right) >\frac{3}{4q_0}.
	\end{align*}
	In particular, we get from \eqref{eq:SL2_field} that $\delta(V\rtimes G,V)> 3(q_0-1)/(4q_0(q-1)) \ge 9/(16(q-1))$ also in this case.
\end{proof}

\begin{proof}[Proof of \Cref{prop:GamL_2}]
Let $G\le \GamL_2(q^{d/2})$ be such that $G\le \GL_d(q)$ is irreducible and primitive. Let $H:=G\cap \GL_2(q^{d/2})$ and $L:=\mathrm F^*(H)$. By Clifford's theorem, $V$ is homogeneous as an $\F_{q^{d/2}}L$-module. If a component is one-dimensional, we have $L\le \Z(H)$ and $\C_H(L)=H$. But $\Z(L)=\C_H(L)$, from which $H=L$ consists of scalars, and since $|G:H|\le d/2$, $G$ cannot be irreducible, false. 
Therefore, $L$ is irreducible on $V$, and so it is absolutely irreducible.
If the layer $E$ of $L$ is nontrivial, then $E=\SL_2(5)$ or $E=\SL_2(q_0)$ with $q_0\ge 4$. Since $G$ normalizes $E$, these cases have been handled in \Cref{l:sl2(5),l:sl2(q_0)}. If $E=1$, then $L=ZQ_8$ with $Z\le \F_{q^{d/2}}^\times$. Then in fact $G$ normalizes $Q_8$,
and so has been handled in \Cref{l:q8}.
\end{proof}

\section{Extraspecial groups}
\label{sec:extraspecial}
\Cref{prop:combination} for extraspecial groups 
is immediate.

\begin{lemma}
\label{l:extraspecial_asymp}
    Let $G\le \GL_d(q)$ be extraspecial of order $r^{2s+1}$ and absolutely irreducible. Then $\alpha(G)>|\Sp_{2s}(r)|/q^{o(d)}$ as $|G|\to \infty$.
\end{lemma}

\begin{proof}
    We have $d=r^s$, and we simply use that the identity has eigenvalue $1$; then 
    \[
\frac{\alpha(G)}{|\Sp_{2s}(r)|}>\frac{1}{|G||\Sp_{2s}(r)|} > \frac{1}{r^{2s+1}r^{4s^2}} > \frac{1}{2^{o(d)}} \ge \frac{1}{q^{o(d)}}
\]
as $d\to \infty$.
\end{proof}

\begin{remark}
\label{rem:extraspecial_problem}
    The assumption that $G$ is sufficiently large is convenient here. For example, if $G=2_-^{11}$ and $q=3$ then $|\Or_{10}^-(2)|/(\alpha(G)|V|^{1/2})$ is larger than $2$ millions. (Note that the nontrivial elements of $2_\pm^{2s+1}$ with eigenvalue $1$ are precisely the noncentral involutions, whose number is seen to be $4^s\pm 2^s-2$.) This says that counting elements of $G$ with eigenvalue $1$ is far from enough to handle all subgroups of $N:=\N_{\GL_d(3)}(G)=G.\Or_{10}^-(2)$.
    
    The argument for subgroups of $N$ can be amended with some more care. However, this
    issue would add (seemingly technical) complications to the arguments of \Cref{sec_primitive_linear}, where we were able to essentially ignore all extraspecial and quasisimple groups of small order and reduce the problem to the generalized Fitting subgroup.
\end{remark}

\section{Quasisimple groups}
\label{sec:quasisimple}
In this section, $d\ge 3$, $V\cong \F_q^d$,  $G$ is a quasisimple group with $S:=G/\Z(G)$, and $G\le \GL_d(q)$ is absolutely irreducible.

\subsection{Alternating groups and groups of Lie type in non-defining characteristic}  We first isolate the case of the fully deleted permutation module for alternating groups.

\begin{lemma}
\label{l:fully_deleted}
Let $G=A_m$ with $m\ge 5$, and let $V$ denote the fully deleted permutation module. Then, $\alpha(G,V)\ge 1- 2(1+\log(m))/m$.
\end{lemma}

\begin{proof}
Let $g\in G$. It is easy to see that if $g$ has at least three cycles, then $g$ fixes a nonzero vector of $V$. Let us then count these elements.

The proportion of elements of $G$ with at most two cycles is
\begin{equation}
\label{eq:permutation_module}
\frac{2\ind{m~\text{odd}}}{m}  + 2\ind{m~\text{even}}\left(\sum_{1\le i<m/2} \frac{1}{i(m-i)} + \frac{2}{m^2}\right),
\end{equation}
where $\ind{-}$ denotes the indicator function.
 Note that 
\begin{align*}
\sum_{1\le i<m/2} \frac{1}{i(m-i)} + \frac{2}{m^2}  &= \frac{1}{m}\sum_{1\le i<m/2} \left(\frac{1}{i} + \frac{1}{m-i}\right) + \frac{2}{m^2} \\
&= \frac{1}{m}\sum_{1\le i\le m-1} \frac{1}{i} \\
&\le \frac{1+\log(m)}{m},
\end{align*}
where in the last inequality we used a standard upper bound for the partial sum of the harmonic series. The lemma follows then from this estimate and \eqref{eq:permutation_module}.
\end{proof}

Assume in the following lemma that either $S=A_m$ with $m\ge 5$, or $S$ is a group of Lie type of (untwisted) Lie rank $r$, defined over $\F_s$, with $(q,s)=1$. In these cases, as we shall recall in the proof, we have that $d\to \infty$ as $|S|\to \infty$.

\begin{lemma}
\label{l:quasi_asymp}
We have $\alpha(G,V) > |\Out(S)|/q^{o(d)}$ as $|S|\to \infty$.
\end{lemma}

\begin{proof}
 Assume first $S=A_m$. By \Cref{l:fully_deleted}, we may assume that $V$ is not the fully deleted permutation module. Then \cite[Lemma 6.1]{guralnicktiep_noncoprime} gives $d \ge (m^2-5m+2)/2$. Therefore, for $m$ sufficiently large, counting only the identity in $G$, we get
 \[
\alpha(G,V)\ge \frac{1}{m!} > \frac{|\Out(S)|}{2^{o(d)}}
 \]
as wanted.

Assume now $S$ is of Lie type; note that $|G|\le s^{O(r^2)}$ and $|\Out(S)|\ll r\log(s)$.
By \cite[Proposition 5.3.1 and Theorem 5.3.9]{kleidmanliebeck}, we see that $d\ge s^{cr}$ for some positive absolute constant $c$, so
\[
\alpha(G,V)\ge \frac{1}{|G|}\ge \frac{|\Out(S)|}{2^{o(d)}}
\]
also in this case.
\end{proof}

\subsection{Groups of Lie type in defining characteristic} 

Assume now $S$ is a group of Lie type of untwisted rank $r$ defined over $\F_s$, and assume $(q,s)\neq 1$. Assume $q$ is a power of the prime $p$. 

If $|S|$ is sufficiently large, then $G=X_\sigma /Z$, where $X$ is a simple linear algebraic group of simply connected type over $K:=\overline{\F_q}$, $\sigma$ is a Steinberg endomorphism of $X$, and $Z$ is a central subgroup of $X_\sigma:=\{x\in X \mid x^\sigma=x\}$ (see e.g. \cite[Theorem 5.1.4 and Tables 5.1.A-5.1.B]{kleidmanliebeck}). 

The irreducible representations of $X_\sigma$ over $K$ are restriction of representations of $X$, and are described by the theory of \textit{highest weights}. We refer to \cite[Section 5.4]{kleidmanliebeck} or \cite[Section 15]{malletesterman} for the basic aspects of this theory (which, for our purposes, are sufficient).

Denote by $u(-)$ be the proportion of $p$-elements. 

\begin{lemma}
\label{l: unipotent_covering}
\label{l: p_elements}
$u(X_\sigma)=u(G)/|Z|$.
\end{lemma}

\begin{proof} Since $(|Z|,p)=1$, each $p$-element of $G$ lifts to a unique $p$-element of $X_\sigma$.
\end{proof}

Next, we recall a result of Steinberg (see \cite[Proposition 5.1.9]{carterbook}). 
We define the level $s_0$ of a Steinberg endomorphism $\sigma$ of $X$ as the absolute value of the eigenvalues of $\sigma$ on $\Gamma\otimes_{\mathbf Z}\mathbf C$, where $\Gamma$ is the character group; see \cite{malletesterman}. If $X_\sigma$ is a Suzuki or Ree group then $s_0$ is not an integer and $s_0^2$ is an integer. In all other cases, $s_0$ is an integer.

This also gives an interpretation of the sentence ``$S$ is defined over $\F_s$'', appearing in the first paragraph of this section -- this holds if and only if either $s_0=s$, or $X_\sigma$ is a Suzuki or Ree group and $s=s_0^2$. 

\begin{theorem} (Steinberg)
\label{p: regular_unipotent}
Assume that $X$ has rank $r$, and let $\sigma$ be a Steinberg endomorphism of $X$ of level $s_0$. Then, the proportion of regular unipotent elements of $X_\sigma$ is $1/s_0^r$.
\end{theorem}

Let $\lambda_1, \ldots, \lambda_r$ be the fundamental dominant weights of $X$, corresponding to the simple roots $\alpha_1, \ldots, \alpha_r$. A character of the form $\sum_{i=1}^r c_i\lambda_i$, with $c_i\in \mathbf Z$, is called \textit{$p$-restricted} if $0\le c_i\le p-1$ for every $i$. An irreducible $KX$-representation $V=V(\lambda)$ is called $p$-restricted if $\lambda$ is $p$-restricted (here $\lambda$ denotes the highest dominant weight of $V$).

If $X_\sigma$ is a Suzuki or Ree group, then the irreducible $KX_\sigma$-modules are the $V(\lambda)$, where $\lambda = \sum_{i=1}^r c_i \lambda_i$, $0\le c_i < s_0^2$, and $\lambda$ is \textit{supported on short roots}, namely, $c_i=0$ if $\alpha_i$ is long. Moreover, these $s_0^r$ representations are pairwise non-equivalent (see \cite[Theorem, Section 20.2]{humphreys2006modular}).  In this case we will also say that $V$ is supported on short roots.

We now record a lemma concerning field of definition of $p$-restricted representations.

\begin{lemma}
\label{l:field_p_restricted}
Let $\sigma$ be a Steinberg endomorphism of $X$ of level $s_0$, and let $d$ be the order of the permutation induced by $\sigma$ on the Dynkin diagram. Let $V$ be an irreducible $p$-restricted $KX_\sigma$-representation, supported on short roots if $X_\sigma$ is Suzuki or Ree, and assume the minimal field of definition for $V$ is $\F_{p^f}$. Then $p^f = s_0$ or $p^f = s_0^d$. 
\end{lemma}

\begin{proof}
Write $s_0=p^e$. Assume first $X_\sigma$ is not Suzuki or Ree.
By \cite[Proposition 5.4.4]{kleidmanliebeck} (note that ``$q$'' in this reference corresponds to ``$s$'' here), 
we have $f \mid ed$, so we only need to show that $f\ge e$. Write $V=V(\lambda)$ where $\lambda= \sum_i c_i\lambda_i$ is the highest dominant weight of $V$, so $0\le c_i\le p-1$ for every $i$. As $X_\sigma$-modules, we have $V(\lambda) \cong V(\lambda)^{(f)}\cong V(p^f\lambda)$. We have $p^f\lambda = \sum_i p^fc_i\lambda_i$, and since $V(\lambda)$ and $V(p^f\lambda)$ are $X_\sigma$-equivalent, by \cite[Theorem 5.4.1]{kleidmanliebeck}, there must exist $c_j$ with $c_jp^f \ge p^e$; so
\[
p^e \le c_jp^f \le (p-1)p^f,
\] 
from which $f\ge e$, as desired. If $X_\sigma$ is Suzuki or Ree, then \cite[Proposition 5.4.4]{kleidmanliebeck} and the same calculation as above gives $f=2e$.
\end{proof}

The following is well-known.

\begin{lemma}
\label{l:weight_zero}
    Assume that $V$ has weight zero. Then every element of $X$ has eigenvalue $1$.
\end{lemma}

\begin{proof}
	Let $T$ be a maximal torus of $V$, so by assumption $\C_V(T)\neq 0$. Now let $g\in X$, and write $g=su$, with $s$ semisimple, $u$ unipotent, and $su=us$. Up to conjugation, we may assume $s\in T$. Since $u$ is a unipotent element normalizing $W:=\C_V(s)$, we have $0\neq \C_W(u)\le \C_V(g)$, which concludes the proof.
\end{proof}

I am grateful to Bob Guralnick for suggesting the proof of the following lemma. See \cite{guralnick_tiep_eigenvalue1} for similar results for algebraic groups and for semisimple elements, which however are not sufficient here. 
\begin{lemma}
\label{l:codimension_1}
    Let $\sigma$ be a Steinberg endomorphism of $X$ of level $s_0$, and let $V$ be an irreducible $KX$-representation. Assume that there exists a weight $\chi$ such that $\chi \sigma \in \mathbf Z \chi$. Then, $\alpha(X_\sigma, V)\ge 1/s_0 + O_{X,\chi}(s_0^{-3/2})$.
\end{lemma}

\begin{proof}
Note that $\{\det(x-1)=0\}$ is a subvariety of $X$ of codimension at most $1$. If it is equal to $X$ then we are clearly done; assume then this is not the case. We will show that there exists a $\sigma$-stable irreducible component $Y$ of $\{\det(x-1)=0\}$ of codimension $1$ in $X$. The Lang--Weyl estimates (which hold for every $\sigma$; see \cite[Section 2]{larsen2024products} and the references therein for the case of Suzuki or Ree groups) will then imply $|Y_\sigma|/|X_\sigma| = s_0^{-1} + O(s_0^{-3/2})$, where the implied constant depends on $Y$. We will see that $Y$ depends on $X$ and $\chi$, so the proof will be concluded.

Let $T$ be a $\sigma$-stable maximal torus, such that $\chi\colon T \to \G_m$, and let $S:=\Ker(\chi)$. By \Cref{l:weight_zero}, $\chi\neq 0$ and so $S$ is a closed subgroup of $T$ of codimension $1$. Note that, since $\chi \sigma \in \mathbf Z \chi$, $S$ is $\sigma$-stable. Denote by $\mathcal R$ the set of regular semisimple elements of $X$, which is $\sigma$-stable. It is known that $\mathcal R$ is open (and dense) in $X$.  

    Assume first that $U:=S^\circ \cap \mathcal R \neq \varnothing$, so $U$ is open and dense in $S^\circ$. Now consider the morphism $f\colon X\times S^\circ \to X$ given by $f(g,s) = s^g$. Then $\text{Im}(f)$ is $\sigma$-stable, and consists of elements having eigenvalue $1$. Next, note that there is a dense subset $D$ of $\text{Im}(f)$  consisting of elements whose fiber has dimension $\dim(T)=\dim(S)+1$. (For example, we can take  $D=\mathcal R \cap \text{Im}(f)$, and we use that $T$ has finite index in $\N_X(T)$.)
     In particular, $\dim(\text{Im}(f))=\dim(X)-1$.
  Setting $Y:=\overline{\text{Im}(f)}$, and noting that $Y$ depends on $X$ and $\chi$, proves the lemma in this case.

    Assume finally that $S^\circ \cap \mathcal R=\varnothing$. Then $L:=\C_X(S^\circ)$ is a reductive group which is not a torus.
    Then, there exists a regular unipotent element $u\in L_0:= [L^\circ,L^\circ]$. 
    Consider then the morphism $f\colon X\times S^\circ \to X$ given by $f(g,s) = (su)^g$. As in the previous case, we see that $\text{Im}(f)$ has codimension $1$ in $X$.
    (In this case, we can choose $D\subseteq \text{Im}(f)$ as the set of elements conjugate to $su$ where $s\in S^\circ$ is such that $\C_X(s)=L$; we recall that $L$ has finite index in $\N_X(L)$, see \cite[Proposition 13.8]{malletesterman}.)
     In particular, letting $Y$ be the closure of the set of elements conjugate to $su'$ where $s\in S^\circ$ and $u'\in L_0$ unipotent, we have that $Y\supseteq \text{Im}(f)$ has codimension $1$, is $\sigma$-stable, and depends only on $X$ and $\chi$, so we are done.
\end{proof}

\subsection{Natural module for classical groups} 
\label{subsec:natural_module}
In this subsection $V$ is the natural module for a classical group $G$. 
For orthogonal groups we assume $n$ even, since for $n$ odd every element of $\SO_n(q)$ has eigenvalue $1$.

For $G=\GL_n(s), \GU_n(s), \Sp_{2n}(s), \Or^\pm_{2n}(s)$, Neumann--Praeger \cite{neumannpraeger} gave estimates for the proportion of elements of $G$ having a given eigenvalue, and showed that its limit as $n\to \infty$ with $s$ fixed exists (and, for example, in non-orthogonal groups it is  $1/s+O(1/s^2)$). See also \cite{fulman2024probabilistic} for recent related results.

We prove an easy result that holds in each coset of $[G,G]$ in $G$. Our result and proof are similar to \cite[Proposition 3.3]{guralnick_tiep_eigenvalue1}, and in particular are an application of the inequality
\begin{equation}
	\label{eq:inequality}
\Big| \bigcup_{g\in G} H^g \Big| \le \frac{|G|}{|\N_G(H):H|}
\end{equation}
for $H\le G$. In \cite{guralnick_tiep_eigenvalue1}, in fact, the authors are interested in semisimple elements with eigenvalue $1$, and  nontrivial estimates for the proportion of semisimple elements of $G$ are required. Here we do not need to focus on semisimple elements, which makes the analysis easier and in some cases allows to handle some more value of $s$ (for $s\ge 3$ the bounds below are meaningful). We give exact bounds (rather than asymptotic), as the proof offers them naturally and they might be useful.

\begin{lemma}
\label{l:natural_module}
    \begin{enumerate}
        \item If 
        $\SL_n(s) \le G\le \GL_n(s)$ with $n\ge 2$ then
        \[
       \frac{1 - \frac{1}{s-1}}{s-1}\le \alpha(G,V) \le \frac{1}{s-1}.
        \]
        \item If $G=\Sp_{2n}(s)$ with $n\ge 2$ then
        \[
       \frac{1 - \frac{1}{s-1}}{s} \le \alpha(G,V) \le \frac{1}{s-1}.
        \] 
        \item  If  $\SU_n(s)\le G\le \GU_n(s)$ with $n\ge 3$ then
        \[
       \frac{1 - \frac{s}{s^2-1}}{s+1} \le \alpha(G,V) \le \frac{s}{s^2-1}.
        \]
        \item  If $G=\Omega^\pm_{2n}(s)$ with $n\ge 4$ and $s$ odd then
        \[
\frac{s(1 - \frac{2s}{s^2-1})}{s^2-1}
 \le \alpha(G,V) \le \frac{2s}{s^2-1}.
        \]
    \end{enumerate}
\end{lemma}

\begin{proof}
	
All proofs are similar; we first prove the upper bound, and deduce the lower bound. The upper bounds in (1)--(3) are contained in \cite[Remark 3.4]{guralnick_tiep_eigenvalue1}. (This is an easy calculation using \eqref{eq:inequality}, applied to the centralizer $H$ of a $1$-space, singular or nondegenerate; in (1) and (3) the bounds are stated only for $\SL_n(s)$ and $\SU_n(s)$ but the  same proof works for any $G$ as in the statement.)
For the upper bound in (4), note that if an element $g$ of $\Omega^\pm_{2n}(s)$, $s$ odd, has eigenvalue $1$, then it fixes either a singular vector, or it centralizes a nondegenerate $2$-space of minus type. (In order to see this, assume that $vg=v$ with $v$ nonsingular. Letting $W=\gen v^\perp$, since $\dim(W)$ is odd and $\det(g)=1$, there exists $0\neq w\in W$ with $wg=w$. If $w$ is singular we are done; otherwise $\gen{v,w}$ is nondegenerate. If it is of minus type we are done again, and otherwise it contains a singular vector, done for the third and final time.) 
The  contributions of the two subgroups in \eqref{eq:inequality} are, respectively, $1/(s-1)$ and $1/(2(s+1))$; bounding $1/(2(s+1)) < 1/(s+1)$ and summing the two contributions we get the desired inequality.

Now let us prove the lower bounds. In (1), for $n=2$ we have $\alpha(G,V)\ge 1/s$,
so assume $n\ge 3$. We count elements that centralize a vector $v$, stabilize a complement and act without eigenvalue $1$ there. Each such element acts without eigenvalue $1$ on exactly one hyperplane (namely $[V,g]$), hence there are no repetitions. Using the upper bound that we just proved, the statement follows from an easy calculation.

The bounds in (2), (3) and (4) are proved in the same way. In (2), we count elements that are regular unipotent on a nondegenerate $2$-space and act without eigenvalue $1$ on the perpendicular complement. (The proportion of regular unipotent elements of $\Sp_2(s)$ is $1/s$.) In (3) we count elements that fix a nondegenerate vector and act without eigenvalue $1$ on the perpendicular complement. In (4) we count elements acting trivially on a nondegenerate $2$-space (of any type) and without eigenvalue $1$ on the orthogonal complement. 
\end{proof}

The previous result is empty for linear and symplectic groups when $s=2$, and does not address orthogonal groups in even characteristic. (This latter exclusion depends on the fact that the centralizer $H$ of a nonsingular vector is a maximal subgroup, and so \eqref{eq:inequality} is useless.) For these cases, we record estimates that can be extracted from \cite{neumannpraeger}.

\begin{lemma}
\label{l:natural_module_q=2}
\begin{enumerate}
    \item For $G=\SL_n(2)$ and $\Sp_{2n}(2)$ we have $\alpha(G,V)> 0.5$.
    \item For $G=\Omega^{\pm}_{2n}(s)$, $s$ even, $n\ge 4$, we have $\alpha(G,V)> 1/s - 1/s^2 - 1/s^{16}$. 
   \end{enumerate}
\end{lemma}

\begin{proof}
(1) Let us start from $G=\SL_n(2)$. By \cite[Theorem 8.1]{neumannpraeger} we get that
    \[
    |1-\alpha(G,V)-G(2,1)|\le c(2) \cdot 2^{-6},
    \]
   where $G(2,1)$ and $c(2)$ are defined in that paper. We have $c(2)<3.5$ by \cite[Lemma 2.2]{neumannpraeger}, and so the error term is at most $0.06$. Moreover, $G(2,1)$ is estimated in \cite[Theorem 4.4]{neumannpraeger}. Looking at the proof, we see that
   \[
   \log(G(2,1)) = -1 - \sum_{k\ge 2}\left(\sum_{j\ge 2; j\mid k} \frac{1}{j}\right) 2^{-k}.
   \]
   In order to upper bound $\log(G(2,1))$, we lower bound the sum considering only the term $k=2$
   and get $\log(G(2,1))< -1 - 1/8$, so $G(2,1) < e^{-1-1/8} <0.33$ and therefore $1-\alpha(G,V) < 0.39$. 
The case of $\Sp_{2n}(2)$, as well as (2), are very similar, and we skip the details.
\end{proof}

\subsection{\Cref{prop:combination} in defining characteristic}
\label{subsec:proof}

We can now prove \Cref{prop:combination} for quasisimple groups $G\le \GL_d(q)$ in defining characteristic, which is the remaining case. We will do this in two lemmas: \Cref{l:p_restricted_bound,l:exceptions}. In \Cref{l:p_restricted_bound}, we will prove a slightly stronger bound in most cases. Recall that $d\ge 3$. 

\begin{lemma}
\label{l:p_restricted_bound}
 We have $\alpha(G, V)\gg q^{-d/3}$, and moreover either $\alpha(G, V)\gg q|\Out(S)|\log(q) q^{-d/3}$, or $G$ and $d$ appear in \Cref{table:exceptions}.
\end{lemma}

\begin{proof}

\textbf{Case 1:} $q\ge s$. 
Recall that by Steinberg's twisted tensor product theorem, if $V$ is not $p$-restricted then
\begin{equation}
\label{eq:twisted}
d\ge d_0^2
\end{equation}
where $d_0$ is the dimension of some $p$-restricted module.

Assume first the rank $r$ is at least $100$.  We have $q|\Out(S)|\log(q)\le q^4$.
By \Cref{l: unipotent_covering,p: regular_unipotent}, the proportion of unipotent elements in $G$ is at least $1/s_0^r \ge 1/q^r$, where $s_0=s$, or $G$ is Suzuki or Ree and $s_0=s^{1/2}$. Since unipotent elements have eigenvalue $1$, we are done if $d/3-4\ge r$, which is to say, $d\ge 3r+12$. Assume then this is not the case; by \cite[Theorem 1.1]{liebeck_ordermax}, $V$ is (a Frobenius twist or the dual of) the natural module. Note that $\alpha(G,V)$ is unchanged by taking duals or Frobenius twists,
so we may assume $V$ is the natural module. Then from \Cref{l:natural_module,l:natural_module_q=2}, $\alpha(G,V)\gg 1/s \gg 1/q$ and we are done.

Assume then $r< 100$, so we may assume $q$ is large and therefore $q|\Out(S)|\log(q) \le q^{1.1}$. In particular, counting regular unipotent elements we are done if $d/3-1.1\ge r$, which is to say,
\begin{equation}
\label{eq:enough_strong}
d \ge 3r+3.3.
\end{equation}
For the cases in \Cref{table:exceptions}, we are done provided
\begin{equation}
\label{eq:enough}
d \ge 3r.
\end{equation}

Assume first that $G$ is exceptional. If $G= G_2(s)$
and $d\ge 10$, we are done by \Cref{eq:enough_strong}. Otherwise, by \cite[Appendix A.49]{lubeck_smallrepn} and \eqref{eq:twisted},
if $G= G_2(s)$ then $d=6$ or $7$. If $d=7$ then $s$ is odd and the representation has weight zero (see \cite[Corollary 2.5]{guralnick_tiep_eigenvalue1}, or just note that $G_2(s)$ embeds into $\SO_7(s)$), so we are done by \Cref{l:weight_zero}. Assume then $s$ is even, so $d=6$. The case $s=q$ appears in \Cref{table:exceptions}, and \Cref{eq:enough} holds. In the case $s<q$, by \Cref{l:codimension_1} we have $\alpha(G,V) \gg 1/s \gg 1/q^{2-1.1}$ and we are done.

Assume now $G$ is exceptional and not $G_2(s)$. If $G$ is not ${}^3D_4(s), {}^2B_2(s)$, then \Cref{eq:enough_strong} follows from \cite[Table 5.4.B]{kleidmanliebeck}. If $G={}^2B_2(s)$ and $d=4$, note that $q\ge s=s_0^2$ and so $1/s_0^2 > q^{-4/3}$, which proves the first bound. Moreover this case appears in  \Cref{table:exceptions}.
If $G={}^2B_2(s)$ and $d>4$ then $d\ge 16$ 
and we are done. If $G= {}^3D_4(s)$, then by \cite[Remark 5.4.7 and Proposition 5.4.8]{kleidmanliebeck} either $V\cong V^{\tau_0}$ and $d\ge 24$, so \Cref{eq:enough_strong} holds, or $q\ge s^3$ and $d\ge 8$,
which is also enough.

Assume now $G$ is classical, and denote by $n$ the dimension of the natural module. We start from the case $X_\sigma=A_r(s)$ or ${}^2A_r(s)$, so $n=r+1$. Note that in this case \Cref{eq:enough_strong} is equivalent to
\begin{equation}
    \label{eq:enough_strong_A}
    d>3n.
\end{equation}
Assume $V$ is (a Frobenius twist or the dual of) the natural module, the alternating square, or the symmetric square. Note that for ${}^2A_r(s)$ we have $q\ge s^2$, unless $r=3$ and $V$ is the alternating square, in which case $q\ge s$.
 We now claim that $\alpha(G,V)\gg_r 1/s$, and one readily checks that this is enough. (Note that some cases appear \Cref{table:exceptions}, namely the ones in the third, fourth, sixth and seventh lines, and the case $d=3$ for $A_1(s)$.) 
The case of the natural module follows from \Cref{l:natural_module,l:natural_module_q=2}. For the symmetric square, we simply count elements with eigenvalue $1$ on the natural module. For the alternating square, we count elements acting as $\text{diag}(\lambda, \lambda^{-1})$ on a $2$-space (nondegenerate for unitary groups) and fixing no $1$-space on a complement, and the claim is proved. (For $A_r(s)$ we could also apply \Cref{l:codimension_1}.)

Assume then $V$ is not among these, so $d\ge n^2/2$ by \cite[Theorem 1.1]{liebeck_ordermax}. Then, \Cref{eq:enough_strong_A} holds provided $n>6$. Assume then $n\le 6$. Since we are excluding natural module, alternating square and symmetric square, we see by \cite{lubeck_smallrepn} and \eqref{eq:twisted} 
that in all the remaining cases either $d>3n$, or $n=2$ and $d=4,5,6$,
or $n=3$ and $d=7,8$.
For $n=3$ and $d=7,8$, the highest dominant weight is stable under the graph automorphism, so by \Cref{l:codimension_1} $\alpha(G,V)\gg 1/s$ and this is enough. For $n=2$ and $d=5$, every element of $G$ has eigenvalue $1$ and we are done.
The cases $n=2$ and $d=4,6$ appear in \Cref{table:exceptions} and we simply count unipotent elements.

The other cases are similar, and easier; let us handle for example $X_\sigma = C_r(q)$, so $n=2r$. If $d \ge n^2/2$ then \Cref{eq:enough_strong} holds for $n\ge 6$, so we may assume $n=4$ or $d<n^2/2$. If $n=4$ then \Cref{eq:enough_strong} holds if $d\ge 10$, and so by \cite{lubeck_smallrepn} we have $d=4$ or $5$. If $d<n^2/2$, then by \cite{liebeck_ordermax} we are reduced to a finite list of possibilities for $V$, and we can apply \Cref{l:codimension_1}. The cases $C_r(s)$, $r=2,3$ and $d=2r$ appear in \Cref{table:exceptions}; the case $C_2(s)\cong B_2(s)$ and $d=5$ does not appear as every element has eigenvalue $1$.

\begin{table}
		\centering
		\caption{Exceptions to the stronger bound in \Cref{l:p_restricted_bound}; see \Cref{rem:table} for more information.}   
		\label{table:exceptions}
		\begin{tabular}{lll}
			\hline\noalign{\smallskip}
			$G$ & d & conditions \\
			\noalign{\smallskip}\hline\noalign{\smallskip}
			$G_2(s)$ & $6$ &  $q=s$ \\
			${}^2B_2(s)$ &  $4$ & $q\ge s$ \\ 
			$A_r(s)$ & $r+1$ & $2\le r\le 5$; $q\ge s$  \\    
   ${}^2A_r(s)$ & $r+1$ & $2\le r\le 3$; $q\ge s^2$ \\
			$A_1(s)$ & $3,4,6$   & $q\ge s^{1/2}$ \\ 
   $A_2(s)$ & $6$ & $q=s$ \\
   $A_3(s),{}^2A_3(s)$ & $6$ & $q=s$ \\
   $C_r(s)$ & $2r$ & $r=2,3$; $q\ge s$ \\
			\noalign{\smallskip}\hline\noalign{\smallskip}
		\end{tabular}
	\end{table}

 \textbf{Case 2:} $q<s=:p^e$. Assume first $G$ is untwisted. Let $\F_{p^f}$ be the minimal field of definition for $V$, so $p^f\le q$.

By \cite[Proposition 5.4.6]{kleidmanliebeck}, $V\otimes K$ is a tensor product of $m\ge e/f$ Frobenius twists of a module $M$. By assuming that $m$ is maximal, we have that the minimal field of definition for $M$ is $\F_s$,
and in particular
\[
|V|\ge p^{df}= p^{\dim(M)^mf}\ge p^{e\cdot \dim(M)^2/2} = s^{\dim(M)^2/2}.
\]

Moreover, $\alpha(G,V)\ge \alpha(G,M)$.
Then, if $M$ does not appear in \Cref{table:exceptions} and $\dim(M)\ge 3$, by Case 1 we readily get $\alpha(G,V)\ge q|\Out(S)|\log(q)q^{-d/3} = q|\Out(S)|\log(q)|V|^{-1/3}$.
Assume then that $\dim(M)=2$ or that $M$ appears in \Cref{table:exceptions}; then by the same calculation we are done unless $X_\sigma=A_1(s)$ and $d=4$.
This case appears in \Cref{table:exceptions} (note that the value $d=4$ was also found in Case 1). 

The case where $G$ is twisted is similar, using again \cite[Proposition 5.4.6]{kleidmanliebeck}.
\end{proof}

\begin{remark}
 \label{rem:table}
    \begin{itemize}
        \item[(i)] In \Cref{table:exceptions}, we did not describe explicitly the representations, and we did not list the precise conditions on $q$ for which the examples are indeed exceptions to the stronger bound $\alpha(G,V)\gg q|\Out(S)|\log(q)q^{-d/3}$ in \Cref{l:p_restricted_bound}. This can be done straightforwardly, and is not relevant for us. Note that we reported the case where $q$ is uniquely determined (e.g. in the first row, we have $q=s$, or the stronger bound holds), and in the other cases, usually the stronger bound holds unless $q$ assumes a couple of values (namely $s$ and $s^2$, or $s^2$ and $s^4$).
     
     \item[(ii)] In the fifth row, for $A_1(s)$, we have either $q\ge s$, or $q\ge s^{1/2}$ and $V=W\otimes W^{(s^{1/2})}$ where $W$ is the natural module. What is more, in the seventh row, $V$ is the alternating square, which can be realized over $\F_s$ also for ${}^2A_3(s)$.
        \end{itemize}
 \end{remark}

 \begin{lemma}
 \label{l:exceptions}
     Assume $V$ appears in \Cref{table:exceptions} and let $G\le T \le \N_{\GL_d(q)}(G)$. If $G$ is sufficiently large then $\alpha(T,V) >2 \log(q)|\Out(S)|
     q^{-d/2}$.
 \end{lemma}

 \begin{proof}
Since $|\Out(S)|\ll \log(s)\ll \log(q)$, it is enough to show that $\alpha(T, V) \ge 1/q^{d/2-\varepsilon}$ for some $\epsilon >0$. Moreover, since $N\le (ZG).\Out(S)$, with $Z=\F_q^\times$, we have $|T:T\cap ZG|\ll \log(q)$ and so we may also assume $G\le T\le ZG$. 

Assume first $q<s$, so by \Cref{rem:table}(ii), $q=s':=s^{1/2}$, $G=A_1(s)$, $d=4$ and $V=W\otimes W^{(s')}$ where $W$ is the natural module. Write $T=GZ'$, where $Z'\le Z$, and work in every coset $z'G$ with $z'\in Z'$; counting elements of $G$ that have eigenvalue $\lambda$ with $\lambda^{s'+1}=z^{-1}$, we deduce that $\alpha(T,V)\gg 1/s' \gg 1/q^{4/2-\eps}$. 

Assume then $q\ge s$. Let $u=2$ if $G={}^2A_r(s)$ with $d\neq 6$, and let $u=1$ otherwise. Then, $q$ is a power of $s^u$. We have $T\le ZG$, and let $H:=G\cap \GL_d(s^u)$, so $|T/H| \le (q-1)/(s^u-1)$ and $\alpha(G,V)\gg \alpha(H,V)s^u/q$. In particular, whenever we show $\alpha(H,V)\gg 1/s^u$ we are done.

For $C_r(s)$ and $r=2,3$, we see that $\alpha(H,V)\gg 1/s$. This follows from the fact that for every $\lambda \in \F_s$, the proportion of elements of $\Sp_{2r}(s)$ ($r=2,3$) with eigenvalue $\lambda$ is $\gg 1/s$.
For $A_r(s)$ and $V$ the natural module, we have  $\alpha(H,V)\gg 1/s$ from \Cref{l:natural_module}. 
For ${}^2A_r(s)$ and $V$ the natural module, setting $K:=H\cap \GU_d(s)$, we deduce from \Cref{l:natural_module} that $\alpha(K,V)\gg 1/s$ and so $\alpha(H,V)\gg 1/s^2=1/s^u$.
For $A_3(s)$ or ${}^2A_3(s)$ and $d=6$, the module is the alternating square; we count elements of $H$ acting as $\text{diag}(\lambda,\lambda^{-1})$, with $\lambda \in \F_s$, on a $2$-space  (nondegenerate for ${}^2A_3(q)$) and fixing no $1$-space on a  complement, and get $\alpha(H,V)\gg 1/s$.
For $A_1(s)$, if an element has eigenvalue $1$ on the natural module then it has eigenvalue $1$ on $V$,
 $\alpha(H,V)\gg 1/s$. For $A_2(s)$ and $d=6$ the module is the symmetric square and we obtain the same bound. 

The remaining cases are $G=G_2(s)$ or ${}^2B_2(s)$. Write $H=GZ'$ where $Z'\le \F_s^\times$. Fix a coset $zG$ with $z\in Z'$. Assuming $|z|>6$ in the case $G=G_2(s)$, there exists a regular semisimple class $C$ in $G$ with eigenvalue $z^{-1}$.
Then $|C|/|G|\asymp 1/s_0^2$ and elements of $zC$ have eigenvalue $1$, so going through all such cosets we get $\alpha(H,V)\gg 1/s_0^2 \gg 1/s^{d/2-\epsilon}$, which is enough. 
 \end{proof}

 \begin{proof}[Proof of \Cref{prop:combination}] This follows immediately from \Cref{l:extraspecial_asymp,l:quasi_asymp,l:p_restricted_bound,l:exceptions}.     
 \end{proof}

\subsection{Summarizing the proof of \Cref{t:affine,t:main,t:subgroup,t:main_2}}
\label{subsec:final}

 The proofs of \Cref{t:affine,t:main,t:subgroup,t:main_2}, are now complete. We summarize them here.

\begin{proof}[Proof of \Cref{t:affine,t:main,t:subgroup,t:main_2}]
\Cref{l:reduction_primitive} reduces \Cref{t:main,t:main_2} to primitive groups $G$. If $G$ is non-affine we apply \Cref{t:luczak_guralnick}. \Cref{t:subgroup} was reduced to primitive affine groups in \cite{bailey2021groups}; so in every case we assume $G$ is primitive affine. It is clear that \Cref{t:affine} implies \Cref{t:main,t:main_2,t:subgroup} in the affine case (see \Cref{l:subgroup_eigenvalue1}), so it suffices to prove \Cref{t:affine}. In \Cref{sec_primitive_linear}, we showed that \Cref{t:affine} follows from \Cref{prop:GamL_1,prop:GamL_2,prop:combination}. We proved \Cref{prop:GamL_1} in \Cref{sec:GammaL1}, \Cref{prop:GamL_2} in \Cref{sec:GamL2}, and we summarized the proof of \Cref{prop:combination} right before \Cref{subsec:final}.
\end{proof}

 \bibliography{references}
\bibliographystyle{alpha}

\end{document}